\documentclass[12pt]{article}
\makeatletter
\newcommand*{\rom}[1]{\expandafter\@slowromancap\romannumeral #1@}

\usepackage{amsmath,amsthm,amssymb,comment,fullpage}

\usepackage[T1]{fontenc}
\usepackage{latexsym}
\usepackage{amsmath,amsfonts,amsthm,amssymb,amscd}
\input amssym.def
\input amssym.tex
\usepackage{color}
\usepackage{hyperref}
\usepackage{url}

\newcommand{\calb}{\mathcal{B}}
\newcommand{\calp}{\mathcal{P}}

\numberwithin{equation}{section}

\newtheorem{thm}{Theorem}[section]

\newtheorem{cor}[thm]{Corollary}
\newtheorem{lem}[thm]{Lemma}
\newtheorem{prop}[thm]{Proposition}
\newtheorem{exa}[thm]{Example}
\newtheorem{defi}[thm]{Definition}

\newtheoremstyle{TheoremNum}
{\topsep}{\topsep}              
{\itshape}                      
{}                              
{\bfseries}                     
{.}                             
{ }                             
{\thmname{#1}\thmnote{ \bfseries #3}}
\theoremstyle{TheoremNum}
\newtheorem{thmrep}{Theorem}

\newtheoremstyle{TheoremNum}
{\topsep}{\topsep}              
{\itshape}                      
{}                              
{\bfseries}                     
{.}                             
{ }                             
{\thmname{#1}\thmnote{ \bfseries #3}}
\theoremstyle{TheoremNum}

\newtheoremstyle{TheoremNum}
{\topsep}{\topsep}              
{\itshape}                      
{}                              
{\bfseries}                     
{.}                             
{ }                             
{\thmname{#1}\thmnote{ \bfseries #3}}
\theoremstyle{TheoremNum}
\newtheorem{correp}{Corollary}

\theoremstyle{plain}

\newtheorem{corollary}[thm]{Corollary}
\newtheorem{definition}[thm]{Definition}

\newtheorem{lemma}[thm]{Lemma}

\newtheorem{rem}[thm]{Remark}

\newtheorem{remark}[thm]{Remark}

\newcommand\be{\begin{equation}}
	\newcommand\ee{\end{equation}}
\newcommand\bea{\begin{eqnarray}}
	\newcommand\eea{\end{eqnarray}}
\newcommand\bi{\begin{itemize}}
	\newcommand\ei{\end{itemize}}
\newcommand\ben{\begin{enumerate}[(a)]}
	\newcommand\een{\end{enumerate}}
\newcommand\bc{\begin{center}}
	\newcommand\ec{\end{center}}
\def\ba#1\ea{\begin{align*}#1\end{align*}}

\newcommand{\C}{\ensuremath{\mathbb{C}}}
\newcommand{\Z}{\ensuremath{\mathbb{Z}}}
\newcommand{\Q}{\mathbb{Q}}
\newcommand{\N}{\mathbb{N}}

\newcommand{\E}{\mathbb{E}}

\newcommand{\leg}[2]{\left( \frac{{#1}}{{#2}} \right)}

\newcommand{\ve}{\varepsilon}

\newcommand{\limi}[1]{\lim_{{#1}\to\infty}}
\newcommand{\lr}[1]{\left\langle#1\right\rangle}
\newcommand{\inv}[1]{{#1}^{-1}}

\newcommand{\norm}[1]{\left\vert\left\vert {#1}\right\vert\right\vert}
\newcommand{\psum}[1]{\frac{1}{p} \sum_{{#1}=1}^{p}}
\newcommand{\nmsum}[1]{\frac{1}{N_m} \sum_{{#1}=1}^{N_m}}
\newcommand{\nsum}[1]{\frac{1}{N} \sum_{{#1}=1}^{N}}

\newcommand{\lpf}[1]{\mathrm{lpf}({#1})}

\newcommand{\clim}[2]{\lim_{{#1}\to\infty} \frac 1{#1} \sum_{{#2}=1}^{#1}}
\newcommand{\dbar}[1]{\overline{d}({#1})}
\newcommand{\fs}{Furstenberg--S\'ark\"ozy }

\usepackage{enumerate}

\date{\today}
\title{The Furstenberg--S\'ark\"ozy Theorem and Asymptotic Total Ergodicity Phenomena in Modular Rings}
\author{Vitaly Bergelson and Andrew Best}

\begin{document}
	\maketitle
	\begin{abstract}
		The \fs theorem asserts that the difference set $E-E$ of a subset $E \subset \N$ with positive upper density intersects the image set of any polynomial $P \in \Z[n]$ for which $P(0)=0$. Furstenberg's approach relies on a correspondence principle and a polynomial version of the Poincar\'e recurrence theorem, which is derived from the ergodic-theoretic result that for any measure-preserving system $(X,\calb,\mu,T)$ and set $A \in \calb$ with $\mu(A) > 0$, one has $c(A):= \lim_{N \to \infty} \frac{1}{N} \sum_{n=1}^N \mu(A \cap T^{-P(n)}A) > 0.$ The limit $c(A)$ will have its optimal value of $\mu(A)^2$ when $T$ is totally ergodic. Motivated by the possibility of new combinatorial applications, we define the notion of asymptotic total ergodicity in the setting of modular rings $\Z/N\Z$. We show that a sequence of modular rings $\Z/N_m\Z, \ m \in \N,$ is asymptotically totally ergodic if and only if $\lpf{N_m}$, the least prime factor of $N_m$, grows to infinity. From this fact, we derive some combinatorial consequences, for example the following.
		Fix $\delta \in (0,1]$ and a (not necessarily intersective) polynomial $P \in \Q[n]$ with $\deg(P) > 1$ such that $P(\Z) \subseteq \Z$. For any integer $N > 1$ with $\lpf N$ sufficiently large, for any subsets $A$ and $B$ of $\Z/N\Z$ such that $|A||B| \geq \delta N^2$, one has $\Z/N\Z = A + B + S$, where $S = \{ P(n) : 1 \leq n \leq N \} \subset \Z/N\Z$.
	\end{abstract}
	
	\section{Introduction}
	

	A nonzero polynomial $P \in \Q[n]$ is \emph{intersective} if $P(\Z)$ contains a multiple of every positive integer; examples are $P(n) = n^2$ and $P(n) = \frac{n(n+1)(n+2)}{6}$, and a nonexample is $P(n)=n^2+1$. The \emph{upper density} of a subset $E \subseteq \N = \{1,2,\ldots\}$ is defined as $\dbar E := \limsup_{N \to \infty} \frac{|E \cap \{1,\ldots, N\} |}{N}$, where $|\cdot|$ denotes the cardinality of a set. We recall the \fs theorem, by now classical:
	\begin{thm}[{\cite{kmf}; see also \cite{sar1,sar3} and \cite{diag77,furstenbergbook}}]\label{fs thm} If $P$ is an intersective polynomial and $E \subseteq \N$ is a subset of natural numbers with $\dbar E > 0$, then $(E-E) \cap P(\Z)$ is nonempty.
	\end{thm}
	Theorem~\ref{fs thm} was proven independently by S\'ark\"ozy in \cite{sar1,sar3} for $P(n) = n^k$ (for integers $k \geq 2$) and $P(n) = n^2 - 1$ using analytic number theory and by Furstenberg in \cite{diag77} for $P(n) = n^2$ and in \cite{furstenbergbook} for $P \in \Q[n]$ such that $P(\Z) \subset \Z$ and $P(0)=0$ using ergodic theory. The unified form of the \fs theorem stated above as Theorem~\ref{fs thm} is due to an observation in \cite{kmf}. Let us describe Furstenberg's ergodic approach to Theorem~\ref{fs thm}. 

	A measure-preserving system is a quadruple $(X,\calb,\mu,T)$, where $(X,\calb,\mu)$ is a probability space and $T: X \to X$ is a measure-preserving transformation. Furstenberg's correspondence principle can be stated as follows: If $E \subseteq \N$ is such that $\overline{d}(E) > 0$, then there is an invertible measure-preserving system $(X,\calb,\mu,T)$ and a set $A \in \calb$ with $\mu(A) = \overline{d}(E) > 0$ such that for all $r \in \N$ and $h_1,\ldots, h_r \in \N$, one has $\overline{d}( E \cap (E - h_1) \cap \cdots \cap (E -h_r)) \ \geq \ \mu(A \cap T^{-h_1}A \cap \cdots \cap T^{-h_r}A)$, where by $E-h_1$ we mean the set $\{e -h_1: e \in E\}$. (See for instance \cite[Theorem 1.1]{b1}.)
	
	Now, let $P \in \Q[n]$ be an intersective polynomial, and suppose $E \subset \N$ has $\dbar E > 0$. If $\dbar{E \cap (E-P(n))} > 0$ for some nonzero $n$, then the set $\{ m \in \N : m, m + P(n) \in E \}$ is nonempty, proving Theorem~\ref{fs thm}. By the Furstenberg correspondence principle, there exists an $(X,\calb,\mu,T)$ and a set $A \in \calb$ such that $\dbar{E \cap (E-P(n))} \geq \mu(A \cap T^{-P(n)}A)$ for any integer $n$. Now one has to show Poincar\'e recurrence along $P$, which asserts that for any measure-preserving system $(X,\calb,\mu,T)$ and any $A \in \calb$ such that $\mu(A) > 0$, there exists a nonzero $n$ such that $\mu(A\cap T^{-P(n)}A) > 0$. The classical Poincar\'e recurrence theorem corresponds to the case $P(n) = n$.
	
	Unlike the classical Poincar\'e recurrence theorem, which can be proved using the pigeonhole principle, Poincar\'e recurrence along $P$ necessitates more sophisticated tools. The approach in \cite{diag77} or \cite{furstenbergbook} relies on an ergodic theorem of the following shape: If $(X, \calb, \mu, T)$ is any measure-preserving system, and $A \in \calb$ is any set with $\mu(A) > 0$, then there exists $c(A) > 0$ such that
	\be \label{positive square eqn 1}
	\clim N n \mu(A \cap T^{-P(n)}A) \ = \ c(A).
	\ee
	It is natural to inquire whether 
	$c(A)$ could be strengthened to an optimally large quantity. 
	If $\mathsf{X}=(X, \calb, \mu, T)$ is ergodic, i.e., if every $T$-invariant set $A \in \calb$ satisfies either $\mu(A) = 0$ or $\mu(A) = 1$, the mean ergodic theorem implies (and is equivalent to) the fact that for any set $A \in \calb$, one has
	\be \label{intro eqn 0}
	\clim N n \mu(A \cap T^{-n}A) \ = \ \mu(A)^2.
	\ee
	If $\mathsf{X}$ is not ergodic, then it follows from the mean ergodic theorem and Cauchy--Schwarz that for any set $A \in \calb$ one has
	\be \label{intro eqn 0'}
	\clim N n \mu(A \cap T^{-n}A) \ \geq \ \mu(A)^2.
	\ee
	See, for instance, the discussion below Theorem 5.1 in \cite{ertau}.
	
	For juxtaposition, suppose $\mathsf{X}$ is totally ergodic, i.e., the system $(X,\calb,\mu,T^k)$ is ergodic for every $k \in \N$. Then certainly \eqref{intro eqn 0} holds with $n$ replaced by $kn$, but more is true. Indeed, 
	total ergodicity is equivalent to ``ergodicity along polynomials''; that is, for every set $A \in \calb$ and every nonconstant integer-valued polynomial\footnote{A polynomial $P \in \Q[n]$ is said to be integer-valued if $P(\Z) \subseteq \Z$.} $P(n)$, we have
	\be\label{intro eqn 1}
	\clim N n \mu(A \cap T^{-P(n)}A) \ = \ \mu(A)^2.
	\ee
	See for instance \cite[Lemma 3.14]{furstenbergbook} or the proof of \cite[Theorem 1.31]{codet}. Following the above discussion of formulas \eqref{intro eqn 0} and \eqref{intro eqn 0'}, one might hope that removing the assumption that $\mathsf{X}$ is totally ergodic would convert \eqref{intro eqn 1} to the inequality
	\be\label{intro eqn 2}
	\clim N n \mu(A \cap T^{-P(n)}A) \ \geq \ \mu(A)^2.
	\ee
	However, this is not the case.
	
	Let us give an example where \eqref{intro eqn 2} fails to hold for some $P$, say $P(n) = n^2$. If $N > 1$ is an integer, then \emph{the rotation on $N$ points} is the ergodic measure-preserving system $(\Z/N\Z, \calp(\Z/N\Z), \mu, T)$, where $\Z/N\Z = \{0,1,\ldots, N-1\}$, $\calp(\Z/N\Z)$ is the power set, $\mu$ is the counting measure normalized so that $\mu(\Z/N\Z) = 1$, and $T$ is the map $n \mapsto n + 1$ modulo $N$. Let $N = 15$ and $A = \{0,7\}$. Then a quick computation shows
	\be
	\clim N n \mu(A \cap T^{-n^2}A) = \frac{1}{15} \sum_{n=1}^{15} \mu(A \cap T^{-n^2}A) = \frac{2}{225} < \frac{4}{225} = \mu(A)^2.
	\ee
	
	Now we consider the bigger picture behind this example. Recall that a measure-preserving system $\mathsf{Y} = (Y,\mathcal{C}, \nu, S)$ is a factor of $(X,\calb,\mu,T)$ if there is a surjective map $\phi: X \to Y$ which preserves the measure (i.e. $\nu = \mu \circ \phi^{-1}$) and intertwines $T$ and $S$ (i.e. $S \circ \phi = \phi \circ T$ holds $\mu$-almost everywhere), and $\mathsf{Y}$ is a finite factor if the underlying set $Y$ has finitely many points. Totally ergodic systems are precisely those which do not have any (nontrivial) finite factors. In this light, one arrives at the conclusion that the reason for \eqref{intro eqn 2} to fail is exactly the presence of finite factors, so the previous example is representative.
	
	There is a natural question that is suggested by the previous discussion. Although the rotation on $N$ points is never totally ergodic, is there a meaningful sense in which it becomes ``more'' totally ergodic as $N$ grows? On account of ``local obstruction'', it is not enough to simply let $N$ grow. In fact, the answer to this question depends on how the factorization of $N$ changes.
	
	Motivated by the fact that total ergodicity is equivalent to the statement that for any positive integer $k$ and any $A \in \calb$,
	\be
	\clim N n \mu(A \cap T^{-kn}A) \ = \ \mu(A)^2,
	\ee
	we define the notion of asymptotic total ergodicity as follows.
	\begin{definition} \label{defn 1} Let $(N_m)$ be a sequence of positive integers. For each $m \in \N$, let $\mathsf{X}_m = (\Z/N_m\Z, \calp(\Z/N_m\Z), \mu_{m}, T_{m})$ be the rotation on $N_m$ points. We say that the sequence $(\mathsf{X}_m)$ is asymptotically totally ergodic if for every positive integer $k$, we have
		\be \label{ATEeqn}
		\limi m \max_{A \subseteq \Z/N_m\Z} \left\vert \nmsum n \mu_m(A \cap T_m^{kn}A) - \mu_m(A)^2 \right\vert \ = \ 0.\footnote{
			In this definition, we use the terminology ``rotation on $N_m$ points'' having in mind the following idea: A circle with $N_m$ points, rotation being $x \mapsto x + 1 \bmod N_m$, seems to be a decent approximation of an \textit{irrational} rotation on the actual unit circle, a transformation which is totally ergodic.
		}
		\ee
	\end{definition}
	For an integer $N > 1$, define $\lpf N$ to be the least prime factor of $N$.
	In the setup of Definition~\ref{defn 1}, the following proposition characterizes the asymptotically totally ergodic sequences $(\mathsf{X}_m)$ as precisely those for which $\limi m \lpf{N_m} = \infty$.
	\begin{prop}\label{ATE for cyclic groups} Assume the setup of Definition~\ref{defn 1}. Then $\limi m \lpf{N_m} = \infty$ if and only if $(\mathsf{X}_m)$ is asymptotically totally ergodic.
	\end{prop}
	\begin{proof} First suppose $\limi m \lpf{N_m} = \infty$. Let $k$ be a positive integer. Then, since the map $n \mapsto kn$ is a permutation of $\Z/N_m\Z$ for all $N_m$ with $\lpf{N_m} > k$, it follows that for all such $N_m$, for all $A \subseteq \Z/N_m\Z$, we have
		\be
		\nmsum n \mu_m (A \cap T_m^{kn} A) \ = \ \nmsum n \mu_m (A \cap T_m^n A) \ = \ \mu_m(A)^2,
		\ee
		where the latter equality holds by either the ergodic theorem or a simple counting argument. 
		
		Second, we show with an example that the condition $\limi M \lpf{N_m} = \infty$ is necessary for the sequence $(\mathsf{X}_m)$ to satisfy the property that for every positive integer $k$, \eqref{ATEeqn} holds.
		If $\limi M \lpf{N_m} \neq \infty$, then, passing to a subsequence if necessary, there is some prime $p$ such that $\lpf{N_{m}} = p$ for all $m$. Take $k = p$. If $\lpf{N_m} = p$, then there are many subsets of $\Z/N_m\Z$ that are $T_m^p$-invariant, and these sets will suffice. For example, if we take $A = \{ n : n \equiv 0 \bmod p\}$, then
		\be
		\nmsum n \mu_m(A \cap T_m^{pn} A) - \mu_m(A)^2 \ = \ \mu_m(A)-\mu_m(A)^2 \ = \ \frac{p-1}{p^2},
		\ee
		which certainly implies \eqref{ATEeqn} does not hold for $k = p$.
	\end{proof}
	Let us upgrade this proposition. We define an asymptotic version of \eqref{intro eqn 1} as follows.
	\begin{definition} \label{defn 2} Let $(N_m)$ be a sequence of positive integers. For each $m \in \N$, let $\mathsf{X}_m = (\Z/N_m\Z, \calp(\Z/N_m\Z), \mu_{m}, T_{m})$ be the rotation on $N_m$ points. Let $P(n)$ be an integer-valued polynomial. We say that the sequence $(\mathsf{X}_m)$ has Property $P$-\emph{LA} (large averages) if
		\be
		\limi m \max_{A \subseteq \Z/N_m\Z} \left\vert \nmsum n \mu_m(A \cap T_m^{P(n)}A) - \mu_m(A)^2 \right\vert \ = \ 0.
		\ee
	\end{definition}
	\begin{remark} \label{canonicality remark}
		There is a small technical issue in the definition of Property $P$-\emph{LA} that we have avoided by referring to the sum over $n \in \{1,\ldots, N_m\}$ rather than the more canonical sum over $n \in \Z/{N_m}\Z$. Namely, if $P(n)$ is an integer-valued polynomial, it is not necessarily true that $\{P(n) : n \in \Z/N\Z \}$ is a well-defined subset of $\Z/N\Z$. For example, if $P(n) = \frac{n(n-1)}{2}$ and $N = 2$, then $P(0) = 0$ but $P(2) = 1$. However, for a given integer-valued polynomial $P$, if $N > 1$ is an integer such that $\lpf N$ is sufficiently large depending on $P$, then $\{P(n) : n \in \Z/N\Z\}$ will be well defined; see Proposition~\ref{canonicality remark proof} for a proof. In relevant results below, since $\limi m \lpf{N_m} = \infty$, we are thus always in the situation where the sum appearing in the definition of Property $P$-\emph{LA} is eventually the more canonical sum over $n \in \Z/N_m\Z$.
	\end{remark}
	For a single measure-preserving system, total ergodicity is characterized by ergodicity along nonconstant integer-valued polynomials \`a la \eqref{intro eqn 1}, not just along polynomials $kn$. Analogously, one hopes that asymptotically totally ergodic sequences $(\mathsf{X}_m)$ of rotations on finitely many points are precisely those that have Property $P$-LA for every nonconstant integer-valued polynomial $P$. Indeed, they are, and this characterization follows from the main theorem of the paper, which we describe now.
	
	If $N$ is a positive integer and $\mu_N$ is the normalized counting measure on $\Z/N\Z$, then for functions $f, g : \Z/N\Z \to \C$ an inner product is defined by
	\be \lr{f,g}_N := \int f \overline{g} \ d\mu_N \ = \ \nsum n f(n)\overline{g(n)},
	\ee
	and we write the corresponding norm $||f||_N := \sqrt{\lr{f,f}_N}$. Since the value of $N$ will be clear from the context, we will generally suppress the $N$ in $\lr{\cdot,\cdot}_N$ and $||\cdot||_N$. Note also that if $T : X \to X$ is a measure-preserving transformation on $(X,\calb,\mu)$, then $T$ defines a unitary operator on $L^2(X)$ by $(Tf)(x) = f(Tx)$.
	
	The following theorem is proven in Section \ref{advanced cyclic group results}:
	
	\begin{thmrep}[\ref{main thm for cyclic groups}]  Let $(N_m)$ be a sequence of positive integers such that $\limi m \lpf{N_m} = \infty$. For each $m \in \N$, let $\mathsf{X}_m$ be the rotation on $N_m$ points. Then, for each nonconstant integer-valued polynomial $P(n)$, we have 
		\be \limi m \sup_{\norm{f_m}\leq 1}\left\vert\left\vert \nmsum n T^{P(n)}_m f_m - \int f_m  \right\vert\right\vert = 0,\ee
		where the supremum is over functions $f_m : \Z/N_m\Z \to \C$ such that $\norm{f_m} \leq 1$.
	\end{thmrep}
	This theorem has several consequences. First, we complete the circle mentioned earlier. By the theorem, any sequence $(\mathsf{X}_m)$ of rotations on $N_m$ points with $\lpf{N_m} \to \infty$ has Property $P$-LA for every nonconstant integer-valued $P$, which of course implies it is asymptotically totally ergodic, which by Proposition~\ref{ATE for cyclic groups} implies $\lpf{N_m} \to \infty$.
	
	Towards another consequence, recall that if $A$ and $B$ are subsets of $\Z/N\Z$, then the sum set $A+B$ is the set $\{a+b: a\in A, b \in B\} \subseteq \Z/N\Z$. We extract from Theorem~\ref{main thm for cyclic groups} the following combinatorial statement, from which other statements follow and which is of independent interest:
	\begin{thmrep}[\ref{quantitative main thm for cyclic groups}] Let $\delta \in (0,1]$ and $P (n)$ be an integer-valued polynomial with degree $d > 1$. There exists a constant $C = C(P,\delta)$ such that the following hold.
		\begin{enumerate}
			\item For any integer $N > 1$ with $\lpf N > C$ and any subsets $A$ and $B$ of $\Z/N\Z$ such that $|A||B| \geq \delta N^2$, the set $A+B$ contains an element of the form $P(m)$ for some $m \in \Z/N\Z$.
			\item More precisely, for any $\ve \in (0,1]$, for any integer $N > 1$, for any subsets $A$ and $B$ of $\Z/N\Z$ such that $|A||B| \geq \delta N^2$, if
			\be
			\lpf N \ > \ \frac{C}{\ve^{2^{d-1}}},
			\ee
			then the number $s$ of pairs $(n,m) \in \Z/N\Z \times \Z/N\Z$ such that $n+P(m) \in A$ and $n \in B$ satisfies $|s-|A||B|| < \ve|A||B|$. 
		\end{enumerate}
	\end{thmrep}
	\noindent In Theorem~\ref{quantitative main thm for cyclic groups}, the first statement follows on taking $\ve = 1$ in the second statement, which quantifies when the number of the pairs $(n,m)$ is close to being as large as one would expect if the polynomial $P$ behaved like a permutation. If we take $B = -A$ in the first statement, we obtain the following corollary as a special case.
	\begin{corollary} Fix $\delta > 0$ and an nonconstant integer-valued polynomial $P(n)$. If $N$ is such that $\lpf{N}$ is sufficiently large and $A \subset \Z/N\Z$ satisfies $|A| \geq \delta N$, then $(A-A) \cap P(\Z)$ is nonempty.
	\end{corollary}
	The above corollary is a finitary, modular version of the \fs theorem, which has the additional benefit of not requiring intersectivity of $P$. This result should be juxtaposed with the polynomial recurrence results in finite fields in \cite[Theorem B]{mccutch} or \cite[Theorem 5.16]{blm}, where intersectivity is essential. 
	Finally, we have the following corollary:
	\begin{correp}[\ref{nice cor 2}] Fix a real number $\delta \in (0,1]$ and an integer-valued polynomial $P(n)$ with $\deg(P) > 1$. For any integer $N > 1$ with $\lpf N$ sufficiently large, for any subsets $A$ and $B$ of $\Z/N\Z$ such that $|A||B| \geq \delta N^2$, one has $\Z/N\Z = A + B + S$, where $S = \{ P(n) : 1 \leq n \leq N \} \subset \Z/N\Z$.
	\end{correp}
	This corollary is only interesting when $\deg(P) > 1$, hence the restriction. It is natural to ask in which cases Corollary~\ref{nice cor 2} is nontrivial and whether it may be strengthened to a form that requires fewer than three sets or that does not require one of the sets to be the image set of an integer-valued polynomial. After the proof of Corollary~\ref{nice cor 2} in Section~\ref{advanced cyclic group results}, we discuss some known cases and several of these hypothetical strengthenings and show why the latter are not possible. It is interesting that even in the case that $A = B$ is an arithmetic progression modulo $N$, adding $S$ is eventually sufficient to cover all of the residues, as if the regular pattern of $A+B$ is ``mixed'' by $S$.
	
	The results of this article complement and partially extend those of $\cite{bbi}$, in which similar problems over finite fields were considered in a more historically minded manner.
	
	The article is structured as follows. In Section~\ref{advanced cyclic group results}, we review preliminary material, give a proof, and describe consequences of Theorem~\ref{main thm for cyclic groups}. In Section~\ref{basic cyclic group results}, we give examples, without the use of Theorem~\ref{main thm for cyclic groups}, of two kinds of sequences $(\mathsf{X}_m)$ that have Property $n^2$-LA. In Section~\ref{counterexamples}, we collect other observations of a finitary nature which give a sharper outline to the phenomenon of asymptotic total ergodicity.

		\section{Modular rings and Property $P$-LA}\label{advanced cyclic group results}
		Before we prove Theorem~\ref{main thm for cyclic groups}, we recall some basic Fourier-analytic facts in the setting of $\Z/N\Z$ and state some additional lemmas.
		\begin{defi} \label{fourier coeff}
			Fix a positive integer $N$ and a function $f : \Z/N\Z \to \C$. Define $\hat{f}(j)$, its \emph{Fourier coefficient at} $j \in \Z/N\Z$, by
			\be
			\hat{f}(j) := \frac 1N \sum_{m=1}^N e^{-2\pi i jm/N} f(m).
			\ee
		\end{defi}
		\begin{prop} Suppose $f$ is as in Definition \ref{fourier coeff}.
			\begin{enumerate}
				\item (Plancherel's theorem) One has \be
				\sum_{j=1}^N \left\vert \hat{f}(j) \right\vert^2 \ = \ \frac 1N \sum_{m=1}^N \left\vert f(m) \right\vert^2.
				\ee
				\item (Fourier inversion formula) For each $m \in \Z/N\Z$, one has
				\be
				f(m) \ = \ \sum_{j=1}^N e^{2\pi i j m/N} \hat{f}(j).
				\ee
			\end{enumerate}
		\end{prop}
		
		\begin{lemma} \label{exp sums lemma}Let $d$, $N$, and $j$ be positive integers such that $j \in \{1,\ldots,N-1\}$. Then
			\be
			\left\vert \frac{1}{N^{d+1}} \sum_{n,h_1,\ldots,h_{d}=1}^N e^{2\pi i n \left( \prod_{i=1}^d h_i\right)j/N} \right\vert \ \leq \ \frac{d}{\lpf N}.
			\ee
		\end{lemma}
		\begin{proof} We argue by induction on $d$. First suppose $d = 1$; for simplicity we'll write $h$ for $h_1$. Fix positive integers $N$ and $j$ such that $j \in \{1,\ldots, N-1\}$. Recall that if $H \in \Z$, then
			\be
			\sum_{n=1}^N e^{2\pi i n H / N} = \begin{cases} 0 & \text{if } H \not\equiv 0 \bmod N \\ N & \text{if } H \equiv 0 \bmod N \end{cases}.
			\ee
			It follows that
			\be
			\frac{1}{N^2} \sum_{n,h = 1}^N e^{2\pi i n h j/N} \ = \ \frac{1}{N^2} \sum_{\substack{h = 1 \\ h j  \equiv 0 \bmod N}}^N N \ = \ \frac{1}{N} \left| \{ h \in \{1, \ldots, N\} : h j \equiv 0 \bmod N\} \right|.
			\ee
			Write $\delta := \gcd (j,N)$. Let $a$ be the minimal positive integer such that $aj \equiv 0 \bmod N$. 
			Then $a = N/\delta$. Hence $|\{ h \in \{1,\ldots, N\} : hj \equiv 0 \bmod N\}| = |\{a,2a,\ldots, \delta a\}| = \delta$. Writing $N = p_1^{a_1}\cdots p_\ell^{a_\ell}$ with the primes $p_1 < p_2 < \cdots < p_\ell$ in ascending order, we observe that since $1 \leq j \leq N-1$, we have $\delta = \gcd(j,N) \leq N/p_1 = N/\lpf N$. Thus
			\be
			\left\vert \frac{1}{N^2} \sum_{n,h = 1}^N e^{2\pi i n h j/N} \right\vert \ = \ \frac{1}{N} \left| \{ h \in \{1, \ldots, N\} : h j \equiv 0 \bmod N\} \right| \ = \ \frac{\delta}{N} \ \leq \ \frac{1}{\lpf N}.
			\ee
			Now suppose $d \geq 2$ and that the statement holds for $d - 1$. Fix positive integers $N$ and $j$ with $j \in\{1,\ldots, N-1\}$. In particular, the induction hypothesis asserts that
			\be
			\left\vert \frac{1}{N^{d}} \sum_{n,h_1,\ldots,h_{d-1}=1}^N e^{2\pi i n \left( \prod_{i=1}^{d-1} h_i\right) (h_dj) /N} \right\vert \ \leq \ \frac{d-1}{\lpf N}
			\ee
			for any $h_d \in \Z$ such that $h_d j \not\equiv 0 \bmod N$. After some manipulation, we observe that
			\begin{align}
				& \left\vert \frac{1}{N^{d+1}} \sum_{n,h_1,\ldots,h_{d}=1}^N e^{2\pi i n \left( \prod_{i=1}^{d} h_i\right) j/N} \right\vert \\
				& \leq \ \frac{1}{N} \sum_{h_d = 1}^N \left\vert \frac{1}{N^{d}} \sum_{n,h_1,\ldots,h_{d-1}=1}^N e^{2\pi i n \left( \prod_{i=1}^{d} h_i\right) j /N} \right\vert \\
				& = \ \frac{1}{N} \sum_{\substack{h_d = 1 \\ h_d j \equiv 0 \bmod N}}^N 1 \quad + \frac{1}{N} \sum_{\substack{h_d = 1 \\ h_d j \not\equiv 0 \bmod N}}^N \left\vert \frac{1}{N^{d}} \sum_{n,h_1,\ldots,h_{d-1}=1}^N e^{2\pi i n \left( \prod_{i=1}^{d} h_i\right)j/N} \right\vert \\
				& = \ \frac{\left| \{ h_d : h_d j \equiv 0 \bmod N\} \right|}{N} + \frac{1}{N} \sum_{\substack{h_d = 1 \\ h_d j \not\equiv 0 \bmod N}}^N \left\vert \frac{1}{N^{d}} \sum_{n,h_1,\ldots,h_{d-1}=1}^N e^{2\pi i n \left( \prod_{i=1}^{d-1} h_i\right)(h_d j)/N} \right\vert \\
				& \leq \ \frac{1}{\lpf N} + \frac{d-1}{\lpf N} \ \leq \ \frac{d}{\lpf N},
			\end{align}
			completing the proof.
		\end{proof}
		
		\begin{lemma}\label{lpf bound} Let $d$ and $N$ be positive integers and let $f : \Z/N\Z \to \C$. Let $T$ be the map $n \mapsto n+1 \bmod N$. Then
			\be
			\left\vert \frac{1}{N^{d}} \sum_{h_1,\ldots,h_{d} = 1}^{N} \frac 1N \sum_{n=1}^{N} \lr{T^{n\prod_{i=1}^{d} h_i}f,f} \right\vert \ \leq \ \left|\int f \ d \mu\right|^2 + \frac{d}{\lpf N} \norm{f}^2.
			\ee
		\end{lemma}
		\begin{proof} Write $H = H(n,h_1,\ldots, h_d) := n\prod_{i=1}^d h_i$. Justifying steps afterwards, we have
			\begin{align}
				\frac{1}{N^{d}} & \sum_{h_1,\ldots,h_{d} = 1}^{N} \frac 1N \sum_{n=1}^N \lr{T^{n\prod_{i=1}^{d} h_i}f,f} \nonumber \\
				& = \ \frac{1}{N^{d+1}} \sum_{n,h_1,\ldots, h_d = 1}^N \lr{T^{H}f,f}\\
				& = \ \frac{1}{N^{d+1}} \sum_{n,h_1,\ldots,h_{d} = 1}^{N} \frac{1}{N} \sum_{m=1}^N f(m+H)\overline{f(m)} \label{blob 1} \\
				& = \ \frac{1}{N^{d+1}} \sum_{j,n,h_1,\ldots,h_{d} = 1}^{N} \widehat{f}(j) e^{2\pi i H j/N} \left( \frac{1}{N} \sum_{m=1}^N e^{2\pi i j m/N} \overline{f(m)} \right) \label{blob 2} \\
				& = \ \frac{1}{N^{d+1}} \sum_{j,n,h_1,\ldots,h_{d} = 1}^{N} \left\vert \widehat{f}(j) \right\vert^2 e^{2\pi i H j/N} \label{blob 3}\\
				& = \ \left\vert \int f \ d\mu \right\vert^2 + \sum_{j=1}^{N-1}\left\vert \widehat{f}(j) \right\vert^2 \frac{1}{N^{d+1}} \sum_{n,h_1,\ldots,h_{d} = 1}^{N} e^{2\pi i H j/N}. \label{blob 4}
			\end{align}
			To get Equation \eqref{blob 1}, we expand the inner product. 
			To get Equations~\eqref{blob 2}~and~\eqref{blob 3}, we apply the Fourier inversion formula to $f(m+H)$ and change the order of summation to highlight the expression for $\overline{\widehat{f}(j)} = \widehat{f}(-j)$. To get Equation \eqref{blob 4}, we separate out the $j = N$ term and reorder the sums again in anticipation of what comes now. By Lemma \ref{exp sums lemma} and Plancherel's theorem, we conclude that
			\begin{align}
				\Bigg\vert \frac{1}{N^{d}} & \sum_{h_1,\ldots,h_{d} = 1}^{N} \frac 1N \sum_{n=1}^{N} \lr{T^{n\prod_{i=1}^{d} h_i}f,f} \Bigg\vert \ \leq \ \left\vert \int f \ d\mu \right\vert^2 + \sum_{j=1}^{N-1} \left\vert \widehat{f}(j) \right\vert^2 \left\vert \frac{1}{N^{d+1}} \sum_{n,h_1,\ldots,h_{d} = 1}^{N} e^{2\pi i H j/N}  \right\vert \nonumber \\
				& \leq \ \left\vert \int f \ d\mu \right\vert^2 + \frac{d}{\lpf N} \sum_{j=1}^N \left\vert \widehat{f}(j) \right\vert^2 \\
				& = \ \left\vert \int f \ d\mu \right\vert^2 + \frac{d}{\lpf N} \left( \frac 1N \sum_{m=1}^N \left\vert f(m) \right\vert^2 \right) \\
				& = \ \left\vert \int f \ d\mu \right\vert^2 + \frac{d}{\lpf N} \norm{f}^2,
			\end{align}
			as desired.
		\end{proof}
		
		We will need to difference a polynomial until it is linear. To this end, we first define some notation: For $a : \Z \to \Z$ and $h \in \Z$, define $\Delta_1 (a(x);h) := a(x+h) - a(x)$ and inductively define $\Delta_{j+1} (a(x); h_1,\ldots, h_{j+1}) := \Delta_1 ( \Delta_j (a(x);h_1,\ldots, h_j); h_{j+1})$. Now, if we are given a polynomial $P \in \Z[n]$ and parameters $h_1, h_2, \ldots, h_{k-1} \in \N$, it follows that the expressions $\Delta_1(P(n);h_1)$, $\Delta_2(P(n);h_1,h_2)$, ..., $\Delta_{k-1}(P(n);h_1,h_2, \ldots, h_{k-1})$ are respectively the degree at most $k-1$, degree at most $k-2$, ..., and degree at most $1$ polynomials in $n$ obtained by differencing $P(n)$ by $h_1$, then the result by $h_2$, and so on, down to a polynomial of degree at most 1.
		\begin{lemma} \label{vdc lemma} Let $N > 1$ be an integer. For any function $f : \Z/N\Z \to \C$ with $||f|| \leq 1$ and any function $P : \Z \to \Z$ satisfying $P(n+N) \equiv P(n) \bmod N$ for any integer $n$,
			\be \label{induction 1}
			\left\vert\left\vert \nsum n T^{P(n)}f \right\vert\right\vert^{2^d} \ \leq \ \frac{1}{N^d} \sum_{h_1,\ldots,h_d = 1}^N  \lr{\nsum n T^{\Delta_d(P(n);h_1,\ldots,h_d)}f , f}.
			\ee
		\end{lemma}
		\begin{proof} We induct on $d$. Suppose $d = 1$. Then
			\begin{align}
				\left\vert\left\vert \nsum n T^{P(n)}f \right\vert\right\vert^2 \ & = \lr{\nsum {n'} T^{P(n')}f,\nsum n T^{P(n)}f} \label{label me 2}\\
				& = \ \nsum {n'} \nsum n \lr{T^{P(n')}f,T^{P(n)}f} \\
				& = \ \nsum {h_1} \nsum n \lr{T^{P(n+h_1)-P(n)}f,f} \\
				& = \ \nsum {h_1} \lr{\nsum n T^{\Delta_1(P(n);h_1)}f,f} \label{label me 3}.
			\end{align}
			Suppose Inequality \eqref{induction 1} holds for $d$. Then
			\begin{align}
				\left\vert\left\vert \nsum n T^{P(n)}f \right\vert\right\vert^{2^{d+1}} \ & = \ \left( \left\vert\left\vert \nsum n T^{P(n)}f \right\vert\right\vert^{2^d}\right)^2 \\
				& \leq \ \left( \frac{1}{N^d} \sum_{h_1,\ldots,h_d = 1}^N \lr{\nsum n T^{\Delta_d(P(n);h_1,\ldots,h_d)}f , f} \right)^2 \\
				& \leq \ \frac{1}{N^d} \sum_{h_1,\ldots,h_d = 1}^N \left\vert\left\vert \nsum n T^{\Delta_d(P(n);h_1,\ldots,h_d)}f \right\vert\right\vert^2 ||f||^2 \\
				& \leq \ \frac{1}{N^d} \sum_{h_1,\ldots,h_d = 1}^N \left\vert\left\vert \nsum n T^{\Delta_d(P(n);h_1,\ldots,h_d)}f \right\vert\right\vert^2 \\
				& = \ \frac{1}{N^d} \sum_{h_1,\ldots,h_d = 1}^N \nsum {h_{d+1}} \lr{\nsum n T^{\Delta_1(\Delta_{d}(P(n);h_1,\ldots,h_d);h_{d+1})}f,f} \\
				& = \ \frac{1}{N^{d+1}} \sum_{h_1,\ldots,h_{d+1} = 1}^N  \lr{\nsum n T^{\Delta_{d+1}(P(n);h_1,\ldots,h_{d+1})}f , f},
			\end{align}
			which proves that Inequality \eqref{induction 1} holds for $d+1$.
		\end{proof}
		
		We now prove the main theorem:
		\begin{thm}\label{main thm for cyclic groups} Let $(N_m)$ be a sequence of positive integers such that $\limi m \lpf{N_m} = \infty$. For each $m \in \N$, let $\mathsf{X}_m$ be the rotation on $N_m$ points. Then, for each nonconstant integer-valued polynomial $P(n)$, we have 
			\[ \limi m \sup_{\norm{f_m}\leq 1}\left\vert\left\vert \nmsum n T^{P(n)}_m f_m - \int f_m  \right\vert\right\vert = 0,\]
			where the supremum is over functions $f_m : \Z/N_m\Z \to \C$ such that $\norm{f_m} \leq 1$.
		\end{thm}
		In the proof of this theorem, we will write $T$ for the map $n \mapsto n + 1$ modulo $N$, where the value of $N$ will always be clear from the context.
		\begin{proof}[Proof of Theorem~\ref{main thm for cyclic groups}]
			It suffices to show the statement of the theorem in the case that the $f_m$'s satisfy $\int f_m \ d\mu_m = 0$.
			
			Fix a polynomial $P \in \Q[n]$ that maps integers to integers. If $P$ has degree 1, then $P(n) = c_1n + c_0$ for some integers $c_1,c_0$, and hence the argument in the proof of Proposition \ref{ATE for cyclic groups} applies (since of course $n \mapsto n + c_0$ is a permutation of $\Z/N_m\Z$ for any $m$). Thus, suppose $P$ has degree $k \geq 2$. Without loss of generality, we may assume the coefficients are actually integers. Indeed, if the theorem statement holds for polynomials with integer coefficients and $P$ has at least one non-integer coefficient, fix some constant $c > 0$ such that $cP \in \Z[n]$. Let $N$ be such that $\lpf{N} > c$. Then the map $\varphi_c(m) := \inv c m$ permutes $\Z/N\Z$, so for any $f : \Z/N\Z \to \C$, it follows that
			\begin{align}
				\norm{\frac 1N \sum_{n=1}^N T^{P(n)} f }^2 \ & = \ \frac{1}{N^2} \sum_{n,n'=1}^N \lr{T^{P(n)}f,T^{P(n')}f} \\
				& = \ \frac{1}{N^2} \sum_{n,n'=1}^N \int f(\inv{c}(cm+cP(n)))\overline{f(\inv{c}(cm+cP(n')))} \ d\mu(m) \\
				& = \ \frac{1}{N^2} \sum_{n,n'=1}^N \lr{T^{cP(n)}(f \circ \varphi_c),T^{cP(n')}(f \circ \varphi_c)} \\
				& = \ \norm{\frac 1N \sum_{n=1}^N T^{cP(n)}(f\circ \varphi_c) }^2.
			\end{align}
			Moreover, we have $\int f\circ \varphi_c \ d\mu \ = \ \int f \ d\mu$, and $\norm{f} \leq 1$ if and only if $\norm{f \circ \varphi_c} \leq 1$. Since $\limi m \lpf{N_m} = \infty$, the previous statements imply
			\be
			\limi m \sup_{\norm{f_m}\leq 1}\left\vert\left\vert \nmsum n T^{P(n)}_m f_m \right\vert\right\vert \ = \ \limi m \sup_{\norm{f_m}\leq 1}\left\vert\left\vert \nmsum n T^{cP(n)}_m f_m \right\vert\right\vert.
			\ee
			
			Thus, write $P(n) = c_kn^k+c_{k-1}n^{k-1} + \cdots + c_1n + c_0$ with integers $c_0, \ldots, c_k$. Fix a positive integer $N$ and a function $f : \Z/N\Z \to \C$ with $\int f \ d\mu = 0$ and $\norm{f} \leq 1$.
			We seek to bound the expression $\left\vert\left\vert \nsum n T^{P(n)}f \right\vert\right\vert$.
			Applying Lemma~\ref{vdc lemma} with $d = k-1$, we have
			\be\label{label me 10}
			\left\vert\left\vert \nsum n T^{P(n)}f \right\vert\right\vert^{2^{k-1}} \ \leq \ \frac{1}{N^{k-1}} \sum_{h_1,\ldots,h_{k-1} = 1}^N  \lr{\nsum n T^{\Delta_{k-1}(P(n);h_1,\ldots,h_{k-1})}f , f}.
			\ee
			We know that $\Delta_{k-1}(P(n);h_1,\ldots,h_{k-1})$ is a polynomial in $n$ of degree at most 1, but it would help to determine it more precisely. When differencing $P(n)$, by induction one can show 
			that
			\begin{align}
				\Delta_{k-1}(P(n);h_1,\ldots,h_{k-1}) \ & = \ k!c_k \left( \prod_{i=1}^{k-1}h_i \right) n + \left(\prod_{i=1}^{k-1}h_i\right) \left( c_{k-1} (k-1)! + \frac{k!}{2} c_k \sum_{i=1}^{k-1} h_i \right) \nonumber \\
				& = \ \left( \prod_{i=1}^{k-1} h_i \right) \left( k!c_k n + \left( c_{k-1} (k-1)! + \frac{k!}{2} c_k \sum_{i=1}^{k-1} h_i\right) \right). \label{label me 8}
			\end{align}
			
			Note that the map $n \mapsto Cn$ is a permutation of $\Z/N\Z$ if and only if $C \in (\Z/N\Z)^\times$. If $N$ is such that $\lpf{N} > \max\{k,\text{the largest prime divisor of } |c_k|\}$, it follows that $k!c_k$ has a multiplicative inverse in $\Z/N\Z$; hence, for any choice of $h_i$'s, the map
			\be \phi_{h_1,\ldots, h_{k-1}}(n) := (k!c_k)^{-1}\left( n - \left( c_{k-1} (k-1)! + \frac{k!}{2} c_k \sum_{i=1}^{k-1} h_i\right)\right),
			\ee
			being a composition of permutations of $\Z/N\Z$, is a permutation of $\Z/N\Z$. By Equation~\eqref{label me 8}, we conclude, for any choice of $h_i$'s, that
			\be
			\Delta_{k-1}(P(n);h_1,\ldots,h_{k-1}) \circ \phi_{h_1,\ldots,h_{k-1}} \ = \ n \prod_{i=1}^{k-1} h_i.
			\ee
			Thus, by using the permutation $\phi_{h_1,\ldots,h_{k-1}}$ to reindex the sum over $n$ on the right-hand side of Inequality~\eqref{label me 10}, we obtain
			\be
			\left\vert\left\vert \nsum n T^{P(n)}f \right\vert\right\vert^{2^{k-1}} \ \leq \ \frac{1}{N^{k-1}} \sum_{h_1,\ldots,h_{k-1} = 1}^N  \lr{\nsum n T^{n\prod_{i=1}^{k-1} h_i}f , f}
			\ee
			Applying Lemma~\ref{lpf bound} with $d = k-1$, we have that
			\be\label{label me 9}
			\left\vert\left\vert \nsum n T^{P(n)}f \right\vert\right\vert \ \leq \ \left( \frac{k-1}{\lpf N} \right)^{2^{-(k-1)}}.
			\ee
			Note that the right-hand side of Inequality \eqref{label me 9} tends to 0 as $\lpf N \to \infty$ uniformly in $f$. The statement of the theorem follows.
		\end{proof}
		
		\begin{rem} For the interested reader, Theorem \ref{main thm for cyclic groups} has a generalization for measure-preserving systems. Let $\mathsf{X}_m := (X_m, \calb_m, \mu_{m}, T_{m} ), m \in \N,$ be a sequence of ergodic measure-preserving systems. Suppose that for each polynomial $P \in \Z[n]$, one has
			\be
			\limi m \sup_{A_m,B_m \in \calb_m} \left\vert \limi{N} \nsum n \mu_m(A_m \cap T^{P(n)}_m B_m) - \mu_m(A_m)\mu_m(B_m) \right\vert \ = \ 0.
			\ee
			Then
			\be \limi m \min\{ d \in \N : T_m^d 
			\text{ is not ergodic} \} = \infty.
			\ee
			The proof of Theorem~\ref{main thm for cyclic groups} is effectively a sketch of a proof of this fact; one just needs Herglotz's theorem. The converse also holds.
		\end{rem}
		Extracting a precise quantitative statement from the proof of the theorem, we obtain the following version of Theorem~\ref{main thm for cyclic groups}. Afterwards, we give a slightly weaker but more readable version, which we will then apply.
		\begin{thm} \label{too precise} Let $\ve \in (0,1]$. Let $N > 1$ be an integer and let $A,B$ be subsets of $\Z/N\Z$. Let $P(n)$ be an integer-valued polynomial with degree $d > 1$ and leading coefficient $c$, and let $c'$ be the smallest positive integer so that $c'P$ has integer coefficients. Let $C_P$ be the maximum of $c'$, $d$, and the largest prime dividing $cc'$. If
			\be
			\lpf N \ > \ \max\left\{C_P, (d-1)\mu(A)(1-\mu(A))\left(\ve \mu(A) \sqrt{\mu(B)}\right)^{-2^{d-1}}\right\},
			\ee
			then the number $s$ of pairs $(n,m) \in \Z/N\Z \times \Z/N\Z$ such that $n+P(m) \in A$ and $n \in B$ satisfies $|s-|A||B|| < \ve|A||B|$.
		\end{thm}
		\begin{proof}
			We observe that $s = \sum_{n=1}^N |B \cap T^{-P(n)}A|$. Hence
			\[ |s - |A||B|| \ = \ \left| \left( \sum_{n=1}^N |B \cap T^{-P(n)}A|\right) - |A||B| \right|. \]
			We want to derive a bound of the shape $|s - |A||B|| < \ve |A||B|$. Dividing through by $N^2$, this means we instead want to derive a bound of the shape
			\be\label{too precise 1}
			\left| \frac{1}{N} \sum_{n=1}^N \mu(B \cap T^{-P(n)}A) -\mu(A)\mu(B) \right| \ < \ \ve \mu(A)\mu(B).
			\ee
			Let us bound the left-hand side of \eqref{too precise 1} using the argument of Theorem~\ref{main thm for cyclic groups}.
			Rewriting the left-hand side of \eqref{too precise 1} using our inner product and applying Cauchy--Schwarz, we find
			\ba \left| \frac{1}{N} \sum_{n=1}^N \mu(B \cap T^{-P(n)}A) -\mu(A)\mu(B) \right| \ & = \ \left| \lr{\frac{1}{N} \sum_{n=1}^N T^{P(n)}(1_A-\mu(A)),1_B} \right| \\ & \leq \ \norm{\frac{1}{N} \sum_{n=1}^N T^{P(n)}(1_A-\mu(A))} \norm{1_B} \\
			& = \ \norm{\frac{1}{N} \sum_{n=1}^N T^{P(n)}f} \sqrt{\mu(B)}
			\ea
			for a function $f = 1_A - \mu(A)$ with integral zero that moreover satisfies $\norm{f}^2 = \mu(A)(1-\mu(A))$. Arguing as in the proof of Theorem~\ref{main thm for cyclic groups} but applying Lemma~\ref{lpf bound} with more attention to the value of $\norm{f}^2$, we obtain a bound on $\norm{\frac{1}{N} \sum_{n=1}^N T^{P(n)}f}$ which implies that
			\be\label{too precise 2}
			\left| \frac{1}{N} \sum_{n=1}^N \mu(B \cap T^{-P(n)}A) -\mu(A)\mu(B) \right| \ \leq \ \left( \frac{d-1}{\lpf N} \mu(A)(1-\mu(A)) \right)^{2^{-(d-1)}}\sqrt{\mu(B)}.
			\ee
			Thus, the statement we are trying to prove now should follow from requiring the right-hand side of \eqref{too precise 2} to be less than the right-hand side of \eqref{too precise 1}. On rearranging the inequality
			\[ \left( \frac{d-1}{\lpf N} \mu(A)(1-\mu(A)) \right)^{2^{-(d-1)}}\sqrt{\mu(B)} \ < \ \ve \mu(A)\mu(B), \]
			we observe that it is equivalent to the inequality $\lpf N > (d-1)\mu(A)(1-\mu(A))\left(\ve \mu(A) \sqrt{\mu(B)}\right)^{-2^{d-1}}$, which does it. The other requirements on $\lpf N$ arise from the argument in the proof of Theorem~\ref{main thm for cyclic groups}.
		\end{proof}
		
		\begin{thm} \label{quantitative main thm for cyclic groups} Let $\delta \in (0,1]$ and $P (n)$ be an integer-valued polynomial with degree $d > 1$. There exists a constant $C = C(P,\delta)$ such that the following hold.
			\begin{enumerate}
				\item For any integer $N > 1$ with $\lpf N > C$ and any subsets $A$ and $B$ of $\Z/N\Z$ such that $|A||B| \geq \delta N^2$, the set $A+B$ contains an element of the form $P(m)$ for some $m \in \Z/N\Z$.
				\item More precisely, for any $\ve \in (0,1]$, for any integer $N > 1$, for any subsets $A$ and $B$ of $\Z/N\Z$ such that $|A||B| \geq \delta N^2$, if
				\be
				\lpf N \ > \ \frac{C}{\ve^{2^{d-1}}},
				\ee
				then the number $s$ of pairs $(n,m) \in \Z/N\Z \times \Z/N\Z$ such that $n+P(m) \in A$ and $n \in B$ satisfies $|s-|A||B|| < \ve|A||B|$. As a remark, one may take $C = \max\{C_P, \frac{d-1}{4} \cdot \delta^{-2^{d-2}}\}$, where $C_P$ is as in Theorem~\ref{too precise}.
			\end{enumerate}
		\end{thm}
		A corollary is as follows.
		
		\begin{cor} \label{nice cor 2} Fix a real number $\delta \in (0,1]$ and an integer-valued polynomial $Q(n)$ with $\deg(Q) > 1$. For any integer $N > 1$ with $\lpf N$ sufficiently large, for any subsets $A$ and $B$ of $\Z/N\Z$ such that $|A||B| \geq \delta N^2$, one has $\Z/N\Z = A + B + S$, where $S = \{ Q(n) : 1 \leq n \leq N \} \subset \Z/N\Z$.
		\end{cor}
		\begin{proof} For each integer $c$, apply Theorem~\ref{quantitative main thm for cyclic groups} with $P = Q - c$ and $\ve = 1$ to the sets $-A = \{-x \in \Z/N\Z : x \in A\}$ and $B$. The constant $C$ does not depend on $c$, so the result follows.
		\end{proof}
		On the one hand, some cases of Corollary~\ref{nice cor 2} are already known. As a trivial example, if $B = -A$ and $\delta > \frac{1}{4}$, then $A + B$ already covers $\Z/N\Z$ since $(-A+c) \cap A \neq \emptyset$. For another example, fix a prime $N=p$ and polynomials $F_1, F_2, F_3$ with integer coefficients. Since we are concerned with sufficiently large $p$, we may assume that $\deg(F_i) < p$ and $\gcd(\deg(F_i),p) = 1$ for each $i$. Let $A$, $B$, and $S$ respectively be the images modulo $p$ of the polynomials $F_1$, $F_2$, and $F_3$. For a fixed integer $c$, the number of solutions to the equation
		\be \label{congruence eqn 1}
		F_1(x_1)+F_2(x_2)+F_3(x_3) \equiv c \bmod p 
		\ee
		is counted by
		\be
		\psum j \sum_{x_1,x_2,x_3 = 1}^p e^{2\pi i j(F_1(x_1)+F_2(x_2)+F_3(x_3)-c)} \ = \ p^2 + \frac{1}{p} \sum_{j=1}^{p-1} e^{-2\pi i j c} \prod_{i=1}^3 \sum_{x_i=1}^p e^{2\pi i j F_i(x_i)}. 
		\ee
		The classical Weil bound (cf. \cite[Theorem 2E]{schmidt}) asserts that $\left\vert \sum_{x=1}^p e^{2\pi i F(x)/p} \right\vert \leq (\deg(F)-1)\sqrt{p}$ for any $F \in \Z/p\Z[x]$ with $\deg(F) < p$ and $\gcd(\deg(F),p)=1$. Applying it here, we conclude that the number of solutions to \eqref{congruence eqn 1} is at least
		\be
		p^2\left(1-\frac{\prod_{i=1}^3 (\deg{F_i} - 1)}{\sqrt p} \right),
		\ee
		which is positive for sufficiently large $p$, which shows that $A + B + S = \Z/p\Z$, which is the same conclusion we could derive with Corollary~\ref{nice cor 2} using a trivial lower bound on $|A||B|$.
		
		On the other hand, there appear to be some nontrivial consequences of Corollary~\ref{nice cor 2}. When $N$ is composite, it becomes more difficult to use either the Weil bound or Corollary~\ref{nice cor 2} to draw conclusions, but in certain cases it is still reasonable. For example, suppose $F_1(n) = n^{c_1}$, $F_2(n) = n^{c_2}$, and $F_3(n) = n^{c_3}$, where $c_1$, $c_2$, and $c_3$ are integers greater than 1. Still assume $A$, $B$, and $S$ are respectively the images modulo $N$ ($N$ odd) of the polynomials $F_1$, $F_2$, and $F_3$. Recall that, for a positive integer $n$, the Euler totient function $\phi(n)$ gives the number of integers in $\{1,\ldots, n\}$ that are coprime to $n$ and that the little omega function $\omega(n)$ gives the number of distinct prime factors of $n$. By the Chinese remainder theorem, it is easy to see that $|A| = |A \cap \Z/N\Z^\times| + |A \setminus \Z/N\Z^\times| \geq \frac{\phi(N)}{c_1^{\omega(N)}} + |A \setminus \Z/N\Z^\times| \geq \frac{\phi(N)}{c_1^{\omega(N)}}$ and similarly for $|B|$. If we assume $\omega(N) \leq k$, then
		\be
		\frac{\phi(N)}{N} \ \geq \ \prod_{i=1}^{\omega(N)} \left(1-\frac{1}{\lpf{N}}\right) \ \geq \ \left(1-\frac{1}{\lpf{N}}\right)^k,
		\ee
		which tends to 1 as $\lpf N$ tends to infinity. Hence, choosing $Q = F_3$ and $\delta = \frac{9}{10} \cdot (c_1c_2)^{-k}$ and applying Corollary~\ref{nice cor 2}, we conclude that for any integer $N > 1$ such that $\lpf{N}$ is sufficiently large and $\omega(N) \leq k$, we have $\Z/N\Z = A + B + S$. We formalize this result in a slightly more general form as follows.
		
		\begin{cor} \label{nice cor 3} Fix positive integers $k$, $c_1$, $c_2$, all greater than 1, and fix an integer-valued polynomial $Q(n)$ with degree $d > 1$. There exists a constant $C = C(k,c_1,c_2,Q)$ with the following property. For any integer $N > 1$ such that $\omega(N) \leq k$ and $\lpf{N} > C$, we have $\Z/N\Z = \{x^{c_1} + y^{c_2} + Q(z) : x,y,z \in \Z/N\Z\}$.
		\end{cor}
		
		In general, when $n^{c_1}$, $n^{c_2}$, and $Q(n)$ are all different, we are not aware of a proof of Corollary~\ref{nice cor 3} by way of lifting arguments or exponential sum estimates. Using the uniform bound -- see, e.g. \cite{hua1,hua2,hua3} --
		\be \left| \sum_{x=1}^{p^n} e^{2\pi i F(x)/p^n} \right| \leq C' p^{n(1-1/\deg(F))},
		\ee
		where $C'$ is an absolute constant, for integer polynomials $F$ that are nonconstant modulo $p$,
		it would seem to follow only if we allow ourselves more sets than $A$, $B$, and $S$.
		
		In any case, Corollary~\ref{nice cor 2} is not restricted to the case when $A$ and $B$ are images of polynomials, and outside of the case when $N$ is prime and $\delta$ is too large, as described above, it is contributing something new. The following discussion of some hypothetical generalizations of Corollary~\ref{nice cor 3} shows the sharpness of Corollary~\ref{nice cor 3}.
		
		First, for every $\delta \in (0,1]$, is it true that for any integer $N > 1$ such that $\lpf N$ is sufficiently large, if $A$, $B$, and $C$ are subsets of $\Z/N\Z$ such that $|A||B||C| \geq \delta N^3$, then $A+B+C = \Z/N\Z$? The answer is no. As is well known for sumsets of subsets of integers, $A+B+C$ can be small, which happens, for example, when $A=B=C$ and $A$ is an arithmetic progression. This carries over in our situation modulo $N$. Suppose $\delta < 1/1000$, and take $A=B=C$ to be $\{0,1,\ldots, \lceil N/10\rceil \} \subset \Z/N\Z$. Then $|A||B||C| \geq \delta N^3$, but $A+B+C =  \{0,1,\ldots, 3\lceil N/10\rceil \}$, which is certainly not all of $\Z/N\Z$.
		
		There is one more hypothetical strengthening to be considered. For every $\delta \in (0,1]$ and every nonconstant integer-valued polynomial $Q$, is it true that for any integer $N > 1$ such that $\lpf N$ is sufficiently large, if $S$ is the image of $Q$ modulo $N$ and $A$ is a subset of $\Z/N\Z$ such that $|A| \geq \delta N$, then $A+S = \Z/N\Z$? Again the answer is no. Indeed, let $\delta < \tfrac 12$, $Q(n) = n^2$, $N = p^2$ for a prime $p \equiv 3 \bmod 4$, and $A = S = \{n^2 : n \in \Z/N\Z\}$. Suppose $x^2 + y^2 \equiv p \bmod{p^2}$. Reducing modulo $p$, we can observe that $x \not\equiv 0 \bmod p$ and $y \not\equiv 0 \bmod p$ and conclude that $x^2 + y^2 \equiv 0 \bmod p$, which falsely implies that $-1$ is a quadratic residue modulo $p$. Thus $p \not\in A + S$, so $A+S \neq \Z/N\Z$, even though $|A| \geq \delta N$ for sufficiently large $p$.

		For completeness, we mention another corollary of Theorem~\ref{quantitative main thm for cyclic groups}:
		\begin{cor} Fix $k > 1$ an integer. For every sufficiently large $p$, we have $\Z/p\Z = \{ x^k + y^k : x,y \in \Z/p\Z\}$.
		\end{cor}
		This result is well known; see for example \cite{small3}. It is also essentially a special case of \cite[Theorem 4]{bbi}, which deals with finite fields. As a remark, when the modulus is composite, it is not always possible, as shown above when $N = p^2$ for $p\equiv 3 \bmod 4$, even to represent any element of $\Z/N\Z$ as a sum of two squares. Considerations of this type are connected to Waring's problem modulo $N$. See \cite{small1,small2} for a discussion.
		
		Finally, we prove the following fact claimed in Remark~\ref{canonicality remark}.
		
		\begin{prop}\label{canonicality remark proof}
			Let $P(n)$ be an integer-valued polynomial. There exists a constant $C > 0$ such that for any integer $N > 1$, if $\lpf N > C$, then the set $\{ P(n) : n \in \Z/N\Z\} \subset \Z/N\Z$ is well defined.
		\end{prop}
		\begin{proof}
			Any integer-valued polynomial can be uniquely represented as an integer linear combination of binomial coefficients, which is well known. 
			
			Define $P_0(n) := 1$, $P_1(n) := n$, and for $i \geq 2$ the $i$th binomial coefficient $P_i(n) := \binom{n}{i} = \frac{n(n-1)\cdots (n-i+1)}{i!}$. We claim that $P_i(n)$ is well defined over $\Z/N\Z$ when $\lpf N > i$, that is, $P_i(n+N) \equiv P_i(n) \bmod N$ for any integers $n$ and $N$ such that $\lpf N > i$. The claim is trivial when $i = 0$ or $i=1$, so suppose $i \geq 2$. Observe that
			\[
			P_i(N+n)-P_i(n) \ = \ \frac{1}{i!}\left( \prod_{j=n-i+1}^n (N+j) - \prod_{j=n-i+1}^n j \right) \ = \frac{M}{i!}, 
			\]
			where
			\[ M := N^i + N^{i-1}\left(\sum_{n-i+1 \leq j_1 \leq n} j_1 \right) + N^{i-2}\left( \sum_{\substack{n-i+1 \leq j_1,j_2 \leq n \\ j_1 < j_2}} j_1j_2 \right) + \cdots + N \sum_{\substack{n-i+1 \leq j_1,j_2,j_3,\ldots,j_{i-1} \leq n \\ j_1 < j_2 < \cdots < j_{i-1}}}j_1\cdots j_{i-1}. \]
			Since $P_i(\Z) \subset \Z$, it follows that $i!$ divides $M$. It is also clear that $M/N$ is an integer. Thus, since $\gcd(N,i!) = 1$ by the assumption $\lpf N > i$, it follows by Euclid's lemma that $i!$ divides $M/N$. Hence $P_i(N+n)-P_i(n) = \frac{M}{N\cdot i!} \cdot N$ expresses $P_i(N+n)-P_i(n)$ as an integer multiple of $N$, proving the claim.
			
			Now, since $P(n)$ is an integer linear combination of binomial coefficients, it follows by the claim that $\{P(n) : n \in \Z/N\Z\} \subset \Z/N\Z$ is well defined if $\lpf N$ is sufficiently large.
		\end{proof}

		\section{Discussion of asymptotic total ergodicity phenomena}\label{counterexamples}
		Consider a sequence $\mathsf{X}_m, \ m \in \N$, of rotations on $N_m$ points. If $\limi m \lpf{N_m} = \infty$, then by Theorem~\ref{main thm for cyclic groups} and Cauchy--Schwarz, for any nonconstant integer-valued polynomial $P$, one has
		\be \label{sec 4 eqn 1}
		\max_{A,B \subseteq \Z/N_m\Z} \left| \nmsum n \mu_m(A \cap T_m^{P(n)}B) - \mu_m(A)\mu_m(B) \right| \to 0 \quad \text{ as } m \to \infty.
		\ee
		Thus, when the smallest prime factor of $N_m$ tends to infinity, we have, loosely speaking, that $\mu_m(A \cap T^{P(n)}B)$ is of size $\mu_m(A)\mu_m(B)$ on average, a kind of (averaged) asymptotic independence of subsets of $\Z/N_m\Z$.
		
		Otherwise, we have $\limi m \lpf{N_m} \neq \infty$. Passing, if needed, to a subsequence, we can assume there exists a prime $p$ such that $p$ divides every $N_m$. From here, the possible behaviors of the averages in \eqref{sec 4 eqn 1} are quite varied; $\nmsum n \mu_m(A \cap T_m^{P(n)}B)$ can be far away from the desired value $\mu_m(A)\mu_m(B)$.
		
		As a trivial example, if $P(n) = pn$, $A = \{ n \in \Z/N_m\Z : n \equiv 0 \bmod p\}$, and $B = \{ n \in \Z/N_m\Z : n\equiv 1 \bmod p\}$, then $\mu_m(A \cap T_m^{pn} B) = 0$ for all $n$. If $p \equiv 1 \bmod 4$, then we can exploit algebraic facts to find ``pathological'' behavior when $P(n) = n^2$, as follows in two examples:
		\begin{exa}[Underergodicity for $P(n)=n^2$] Let $p$ be prime with $p \equiv 1 \bmod 4$, and let $N = kp$ for some positive integer $k$. Consider the rotation on $N$ points $\mathsf{X} = (\Z/N\Z, \calp(\Z/N\Z), \mu,T)$. Then there is a set $A \subset \Z/N\Z$ of measure $\mu(A) = 2/p$ such that
			\be\label{UEExample} \frac{1}{N}\sum_{n=1}^{N} \mu (A \cap T^{n^2}A) \ = \ \frac 12 \mu(A)^2.
			\ee
		\end{exa}
		\begin{proof} Since $p \equiv 1 \bmod 4$, we may pick $q_1$ and $q_2$ (nonzero) quadratic nonresidues modulo $p$ such that $q_1 \equiv -q_2 \bmod p$. Let $A = \cup_{c=0}^{k-1} \{cp,cp+q_1\}$ and note $\mu(A) = 2k/kp = 2/p$. By construction $A$ is invariant under $T^{p}$. Moreover, for each $n \in \{1,\ldots,N\}$ with $p \nmid n$, we have $A \cap T^{n^2}A = \varnothing$, since any two elements of $A$ differ by $cp$, $cp + q_1$, or $cp+q_2$ (with $c \in \{0,\ldots,k-1\}$). Thus
			\ba
			\frac{1}{N}\sum_{n=1}^{N} \mu (A \cap T^{n^2}A) \ & = \ \frac{1}{N} \sum_{\substack{n=1 \\ p \mid n}}^{N} \mu(A \cap T^{n^2}A) + \frac{1}{N} \sum_{\substack{n=1 \\ p \nmid n}}^{N} \mu (A \cap T^{n^2}A) \\
			& = \ \frac{k}{kp} \cdot \mu(A) + \frac{kp - k}{kp} \cdot \mu (\varnothing) \\
			& = \ \frac{2}{p^2} \ = \ \frac 12 \mu(A)^2.
			\ea 
		\end{proof}
		
		\begin{exa}[Overergodicity for $P(n)=n^2$] Let $p$ be prime with $p \equiv 1 \bmod 4$, and let $N = kp$ for some positive integer $k$. Consider the rotation on $N$ points $\mathsf{X} = (\Z/N\Z, \calp(\Z/N\Z), \mu,T)$. Then there is a set $A \subset \Z/N\Z$ of measure $\mu(A) = 2/p$ such that
			\be\label{OEExample} \frac{1}{N}\sum_{n=1}^{N} \mu (A \cap T^{n^2}A) \ = \ \frac 32 \mu(A)^2.
			\ee
		\end{exa}
		\begin{proof} Since $p \equiv 1 \bmod 4$, we may pick $q_1$ and $q_2$ nonzero quadratic residues modulo $p$ such that $q_1 \equiv -q_2 \bmod p$. For $i \in \{1,2\}$, choosing $m_i \in \{0, \ldots, p-1\}$ such that $m_i^2 \equiv q_i \bmod p$, observe that
			\be
			S_i \ := \ \{ n \in \Z/N\Z : n^2 \equiv q_i \bmod p\} \ = \ \{cp \pm m_i \in \Z/N\Z : 0 \leq c \leq k-1\},
			\ee
			and that $S_i$ has cardinality $2k$. Let $A = \cup_{c=0}^{k-1} \{cp,cp+q_1\}$ and note $\mu(A) = 2k/kp = 2/p$. As before, any two elements of $A$ are either $cp$, $cp + q_1$, or $cp+q_2$ apart (with $c \in \{0,\ldots,k-1\}$), and by construction $A$ is invariant under $T^{p}$. Moreover, for each $n \in \{1,\ldots,N\}$ with $p \nmid n$, we have three possibilities: First, if $n \in S_1$, then $A\cap T^{n^2}A = \cup_{c=0}^{k-1} \{cp + q_1\}$, and note that this set has $\mu$ measure $1/p$. Second, if $n \in S_2$, then $A\cap T^{n^2}A = \cup_{c=0}^{k-1} \{cp\}$, and again this set has $\mu$ measure $1/p$. Otherwise we have $A\cap T^{n^2}A = \varnothing$. Thus
			\ba
			\frac{1}{N}\sum_{n=1}^{N} \mu (A \cap T^{n^2}A) \ & = \ \frac{1}{N} \sum_{\substack{n=1 \\ p \mid n}}^{N} \mu(A \cap T^{n^2}A) + \frac{1}{N} \sum_{\substack{n=1 \\ n \in S_1}}^{N} \mu (A \cap T^{n^2}A) \\
			& + \ \frac{1}{N} \sum_{\substack{n=1 \\ n \in S_2}}^{N} \mu (A \cap T^{n^2}A) + \frac{1}{N} \sum_{\substack{n=1 \\ n \nmid p \\ n \notin S_1 \cup S_2}}^{N} \mu (A \cap T^{n^2}A) \\
			& = \ \frac{k}{kp} \cdot \mu(A) \ + \ \frac{2k}{kp} \cdot \frac 1p \ + \ \frac{2k}{kp} \cdot \frac 1p \ + \ 0 \\
			& = \ \frac{6}{p^2} \ = \ \frac 32 \mu(A)^2.
			\ea
		\end{proof}
		The previous examples are not the most extreme. In general---that is, without any assumption on the residue class of $p$---we have the following remarks. By Lagrange interpolation, an arbitrary function $\Z/p\Z \to \Z/p\Z$ is a polynomial function with integer coefficients, and moreover for any polynomial $P \in \Z[n]$ we have $P(n+p) \equiv P(n) \bmod p$ for any $n$ and $p$. Thus, since $p$ divides every $N_m$, it is possible to arrange for the following:
		\begin{itemize}
			\item There exist a (nonconstant mod $p$) polynomial $P_1 \in \Z[n]$ and $A_m \subseteq \Z/N_m\Z, \ m \in \N$, with $\mu_m(A_m)$ bounded away from zero such that $\mu_m(A_m \cap T_m^{P_1(n)} A_m) \approx \mu_m(A_m)$ for every $n$.
			\item There exist a (nonconstant mod $p$) polynomial $P_2 \in \Z[n]$ and $A_m \subseteq \Z/N_m\Z, \ m \in \N$, with $\mu_m(A_m)$ bounded away from zero such that $\mu_m(A_m \cap T_m^{P_2(n)} A_m) = 0$ for every $n$.
		\end{itemize}
		We now prove more precise formulations of these two claims.
		
		\begin{prop} Let $p > 3$ be prime, and let $N$ satisfy $\lpf N = p$. Consider the rotation on $N$ points $\mathsf{X} = (\Z/N\Z, \calp(\Z/N\Z), \mu,T)$. Then there exists a polynomial $P(n) \in \Z[n]$ and a set $A \subset \Z/N\Z$ with $\mu(A) = \frac{1}{2}+\frac{1}{2p}$ such that $\mu(A \cap T^{P(n)}A) \ \geq \mu(A) - \frac{1}{p}$ for every $n \in \Z/N\Z$. Moreover, $P(n)$ can be taken to be nonconstant mod $p$. 
		\end{prop}
		\begin{proof}
			Let $A = \{2i + np : 0 \leq (p-1)/2, 0\leq n \leq N/p - 1\}$. Then $\mu(A) = \frac{1}{N} \cdot \frac Np \cdot \frac{p+1}{2} = \frac{1}{2} + \frac{1}{2p}$. By construction, $|A \cap T^2A| = (|A|-1)\cdot \frac{N}{p}$. Hence, by $T^p$-invariance of $A$, we observe that $\mu(A\cap T^{i}A) = \frac{1}{2} - \frac{1}{2p}$ for any $i \equiv 2 \bmod p$. Thus, let $P : \Z/p\Z \to \Z/p\Z$ be some function with image $\{0,2\}$. By Lagrange interpolation, $P$ is a polynomial function with integer coefficients. The result follows since $A$ is $T^p$-invariant and $P(n+p) \equiv P(n) \bmod p$ for each $n \in \{0,\ldots,p-1\}$.
		\end{proof}
		
		\begin{prop} Let $p > 3$ be prime, and let $N$ satisfy $\lpf N = p$. Consider the rotation on $N$ points $\mathsf{X} = (\Z/N\Z, \calp(\Z/N\Z), \mu,T)$. Then there exists a polynomial $P(n) \in \Z[n]$ and a set $A \subset \Z/N\Z$ with $\mu(A) = \frac{1}{2}-\frac{3}{2p}$ such that $\mu(A \cap T^{P(n)}A) \ = \ 0$ for every $n \in \Z/N\Z$. Moreover, $P(n)$ can be taken to be nonconstant mod $p$.
		\end{prop}
		\begin{proof}
			Let $k = \frac{p-5}{2}$, and define $A = \{2i + np : 0 \leq i \leq k, 0 \leq n \leq N/p -1 \} \subset \Z/N\Z$. Clearly $|A| = (k+1)\cdot N/p$, so that $\mu(A) = \frac 12 - \frac{3}{2p}$. By construction, if $x,y \in A$, then $x-y \not\equiv i \bmod p$, where $i \in \{1,3\}$. Thus, let $P : \Z/p\Z \to \Z/p\Z$ be some function with image $\{1,3\}$. By Lagrange interpolation, $P$ is a polynomial function with integer coefficients. The result follows.
		\end{proof}

		Approaching the phenomenon of asymptotic total ergodicity from another angle, one may ask whether there is any integer-valued polynomial $P$ such that \eqref{sec 4 eqn 1} holds for every sequence $(N_m)$ of increasing integers. When $\deg(P) = 1$, it is not hard to answer this question---yes if and only if $P(n) = \pm n + c$ for some integer $c$. When $\deg(P) > 1$, the following proposition gives a negative answer.
		\begin{prop}Let $P(n)$ be an integer-valued polynomial of degree $d > 1$. Then there exists a constant $C = C(P) > 0$ and an increasing sequence of integers $(N_m)$ with the following property: Let $(\Z/N_m\Z,\calp(\Z/N_m), \mu_m,T_m)$ be the rotation on $N_m$ points; then, for every $m$, there exist $A_m,B_m \subseteq \Z/N_m\Z$ such that
			\be
			\left\vert \nmsum n \mu_m(A_m\cap T_m^{-P(n)} B_m) - \mu_m(A_m)\mu_m(B_m) \right\vert \ \geq \ C.
			\ee
		\end{prop}
		\begin{proof} Suppose $P$ has degree $d > 1$, and fix a positive integer $c$ such that the polynomial map $Q(n) := P(cn)$ has integer coefficients. By Dirichlet's theorem on primes in arithmetic progressions, choose a prime $p$ such that $p > c$, $d$ divides $p-1$, and $p$ does not divide the leading coefficient of $Q$. By choice of $p$, $Q$ has degree $d$ when viewed as a map $\Z/p\Z \to \Z/p\Z$ with coefficients reduced modulo $p$. Thus, by \cite{dickson} (see also \cite[Corollary 7.5]{LN}), $Q$ is not a permutation of $\Z/p\Z$. Hence there exists $a \in \Z/p\Z$ such that the preimage $Q^{-1}(\{a\})$ in $\Z/p\Z$ has cardinality $m_a \geq 2$. Choose an increasing sequence of integers $N_m$ such that $\lpf{N_m} = p$. For each $m$, let $A_m := \{n \in \Z/N_m\Z : n \equiv 0 \bmod p \}$ and $B_m := \{ n \in \Z/N_m\Z : n \equiv a \bmod p \}$. Then we have
			\be
			A_m \cap T^{-Q(n)}B_m = \begin{cases} A_m \quad \text{if } Q(n) \equiv a \bmod p, \\ \varnothing \quad \text{otherwise.} \end{cases}
			\ee
			Since $\lpf{N_m} = p > c$, the map $n \mapsto cn$ is a permutation of $\Z/N_m\Z$. Also, $Q(p+n) \equiv Q(n) \bmod p$ for each $n$. It follows that
			\be
			\nmsum n \mu_m(A_m \cap T_m^{-P(n)}B_m) \ = \ \nmsum n \mu_m(A_m\cap T_m^{-Q(n)}B_m) \ = \ \frac{m_a}{p} \cdot \mu_m(A_m) \ \geq \ \frac{2}{p^2},
			\ee
			which concludes the proof of the proposition on taking $C = \frac{1}{p^2}$.
		\end{proof}
		

		In conclusion, we address one more possibility for an asymptotic version of a mixing notion.
		A system $\mathsf{X} = (X,\calb,\mu,T)$ is weakly mixing if $T \times T$ is ergodic. Equivalently, $\mathsf{X}$ is weakly mixing if and only if, for every $A \in \calb$, one has
		\be
		\clim N n \left| \mu(A \cap T^{-n}A) - \mu(A)^2 \right| \ = \ 0.
		\ee
		In view of this characterization, one might make asymptotic the notion of weak mixing thus:
		\begin{definition}
			Let $(N_m)$ be a sequence of positive integers. For each $m \in \N$, let $\mathsf{X}_m = (\Z/N_m\Z, \calp(\Z/N_m\Z), \mu_{m}, T_{m})$ be the rotation on $N_m$ points. We say that the sequence $(\mathsf{X}_m)$ is asymptotically weakly mixing if
			\be
			\limi m \max_{A \subseteq \Z/N_m\Z} \nmsum n \left\vert \mu_m(A \cap T_m^{n}A) - \mu_m(A)^2 \right\vert \ = \ 0.
			\ee
		\end{definition}
		\noindent Could a sequence of rotations on $N_m$ points such that $\lpf{N_m} \to \infty$ actually be asymptotically weakly mixing rather than merely asymptotically totally ergodic? The answer is no, as can be seen by considering the ``interval'' $A = \{0,1,\ldots, \lfloor \frac{N_m}{10} \rfloor\}$.

		\section{Illustration of Property $n^2$-LA}\label{basic cyclic group results}
		In this section, we show via direct calculation and without the help of Theorem~\ref{main thm for cyclic groups} that for any positive integer $k$, the sequence of $k$th powers of primes which are congruent to 3 modulo 4 has Property $n^2$-LA. The situation when $k > 1$ suggests the potential difficulty that is overcome by choosing estimation as the proof strategy of Theorem~\ref{main thm for cyclic groups} rather than direct calculation; this section concludes with some more discussion. We start with the straightforward case $k = 1$.

		\begin{prop}\label{p3mod4} Let $p$ be prime with $p \equiv 3 \bmod 4$. Consider the measure-preserving system $\mathsf{X}_p := (\Z/p\Z,\calp (\Z/p\Z), \mu, T)$, where $\mu$ is normalized counting measure and $T$ the map $n \mapsto n + 1$ modulo $p$. Let $a \in \Z/p\Z$ be nonzero. For all $A \subset \Z/p\Z$, we have
			\be\label{basic cyclic group results eqn 1}
			\psum n \mu(A \cap T^{an^2}A) \ = \ \mu(A)^2.
			\ee
		\end{prop}
		\begin{proof} First suppose that $a$ is a quadratic residue mod $p$. Then there exists a nonzero $\rho \in \Z/p\Z$ such that $\rho^2 \equiv a \bmod p$. Then, the change of variables $n \mapsto \inv{\rho}n$ permutes $\Z/p\Z$, so it follows that
			\be
			\psum n \mu(A \cap T^{an^2}A) \ = \ \psum n \mu(A \cap T^{n^2}A).
			\ee
			Thus, we may rewrite the left-hand side of \eqref{basic cyclic group results eqn 1} as 
			\be
			\psum n \mu(A \cap T^{n^2}A) \ = \ \sum_{\substack{i = 0 \\ \leg ip \in\{0,1\}}}^{p-1} \frac{C(i,p)}{p} \cdot \mu(A \cap T^i A),
			\ee
			where $\leg \cdot\cdot$ is the Legendre symbol and $C(i,p)$ is the number of solutions to $x^2 = i$ in $\Z/p\Z$. To simplify this, note that $C(0,p) = 1$ and $C(i,p) = 2$ for all $i$ that are nonzero quadratic residues modulo $p$, since $k$ is a solution iff $p-k$ is. Thus we have
			\be 
			\label{eqn2alt} \sum_{\substack{i = 1 \\ \leg ip \in\{0,1\}}}^{p} \frac{C(i,p)}{p} \cdot \mu(A \cap T^i A) \ = \ \frac 1p \cdot \mu(A) + \sum_{\substack{i = 1 \\ \leg ip = 1}}^{p-1} \frac{2}{p}\cdot \mu(A \cap T^i A).
			\ee
			To conclude, we need only two observations. First, note that the number of nonzero quadratic residues modulo $p$ is $(p-1)/2$, so the right sum in \eqref{eqn2alt} has $(p-1)/2$ summands. Second, since $p \equiv 3 \bmod 4$, $i$ is a nonzero quadratic residue if and only if $p-i$ is not. This, combined with the fact that $\mu(A \cap T^iA) = \mu(A \cap T^{p-i}A)$ by shift-invariance, implies
			\be\label{basic cyclic group results eqn 3}
			\sum_{\substack{i = 1 \\ \leg ip = 1}}^{p-1} \frac{2}{p}\cdot \mu(A \cap T^i A) \ = \  \sum_{\substack{i = 1 \\ \leg ip = 1}}^{p-1} \frac{1}{p}\cdot \left( \mu(A \cap T^i A) + \mu(A \cap T^{p-i} A)\right) \ = \ \frac 1p \sum_{\substack{i = 1 \\ \leg ip \in \{ 1, -1\} }}^{p-1} \mu(A \cap T^{i}A),
			\ee
			so that
			\be \psum n \mu(A \cap T^{an^2}A) \ = \ \psum n \mu(A \cap T^n A) \ = \ \mu(A)^2,
			\ee
			where the last equality follows by simple counting or by the mean ergodic theorem.
			
			Now suppose that $a$ is not a quadratic residue mod $p$. Since the product of a (nonzero) quadratic residue and a (nonzero) quadratic nonresidue of $\Z/p\Z$ is a quadratic nonresidue, and for $c \in \Z/p\Z \setminus \{0\}$ the map $n \mapsto cn$ is a permutation of $\Z/p\Z$, it follows that $n \mapsto an^2$ sends $0$ to $0$ and is a two-to-one map from $\Z/p\Z \setminus \{0\}$ onto the set of (nonzero) quadratic nonresidues. Thus, we may rewrite the left-hand side of \eqref{basic cyclic group results eqn 1} as
			\be
			\psum n \mu(A \cap T^{an^2}A) \ = \ \frac 1p \cdot \mu(A) + \sum_{\substack{i = 1 \\ \leg ip = -1}}^{p-1} \frac{2}{p}\cdot \mu(A \cap T^i A).
			\ee
			The argument now proceeds as in \eqref{basic cyclic group results eqn 3}.
		\end{proof}
		
		We conclude with the second proposition of this section, which pushes the proof technique a little further.
		\begin{prop}\label{p3mod4powers}
			Let $p$ be prime with $p \equiv 3 \bmod 4$, and fix an integer $k > 1$. Consider the measure-preserving system $\mathsf{X}_{p^k}:=(\Z/p^k\Z,\calp (\Z/p^k\Z), \mu, T)$, where $\mu$ is normalized counting measure and $T$ the map $n \mapsto n + 1$ modulo $p^k$. For each $m \in \{0, 1, \ldots, k \}$, define the $\sigma$-algebra of $T^{p^m}$-invariant subsets $\calb_{m} := \{ A \subset \Z/p^k\Z : T^{p^m}A = A\}$. For all $A \subset \Z/p^k\Z$,
			\be\label{pkgoal}
			\frac{1}{p^k} \sum_{n=1}^{p^k} \mu(A \cap T^{n^2}A) \ = \ \mu(A)^2 \ + \ \frac{C_k}{p^{k/2}} \mu(A) \ + \ \sum_{m=1}^{k-1} \frac{(-1)^m}{p^{\lceil m/2\rceil}} \ \langle \E ( 1_A \vert \calb_m ),1_A \rangle, 
			\ee
			where $C_k = 1$ if $k$ is even and 0 if $k$ is odd.
		\end{prop}
		\begin{rem} In particular, note that as $p \to \infty$ along primes congruent to 3 mod 4,
			\be
			\max_{A \subset \Z/p^k\Z} \left\vert \frac{1}{p^k} \sum_{n=1}^{p^k} \mu(A \cap T^{n^2}A) - \mu(A)^2 \right\vert \to 0.
			\ee
			This is a special case of Theorem~\ref{main thm for cyclic groups}. 
		\end{rem}
		\begin{proof}[Proof of Proposition~\ref{p3mod4powers}] As before, we first determine the image of $n \mapsto n^2$ modulo $p^k$, first without multiplicity, then with multiplicity. After this, we compute and simplify the relevant average. We represent an element of $\{0,1,\ldots, p^k -1\}$ as $x = a_0 + a_1p + a_2p^2+\cdots+a_{k-1}p^{k-1}$ with all $a_i \in \{ 0,1,\ldots,p-1\}$. Let us partition our domain into cells:
			\begin{align*} D_0 & \ := \ \{ a_0 + a_1p + a_2p^2+\cdots+a_{k-1}p^{k-1} : a_0 \neq 0\}, \\
				D_1 & \ := \ \{ a_1 \neq 0 \text{ and } a_0 = 0\}, \\
				D_2 & \ := \ \{ a_2 \neq 0 \text{ and } a_0 = a_1 = 0 \}, \\
				& \quad \vdots \\
				D_{\lceil k / 2 \rceil - 1} & \ := \ \{ a_{\lceil k/2 \rceil - 1} \neq 0 \text{ and } a_0 = a_1 = \cdots = a_{\lceil k/2 \rceil-2} = 0 \}, \\
				D_{\lceil k / 2 \rceil} & \ := \ \{a_0 = a_1 = \cdots = a_{\lceil k/2 \rceil -1} = 0 \}.
			\end{align*} Let us define some subsets\footnote{We handle the case of $k$ even and $k$ odd simultaneously until the end. Even though $S_{\lceil k/2 \rceil - 1}$ looks strange for $k$ odd, we mean it.} of the range:
			\begin{align*}
				S_0 & \ := \ \{ q + c_1p+c_2p^2 + \cdots + c_{k-1}p^{k-1} : c_i \in \Z/p\Z, \; q \neq 0 \text{ a quadratic residue mod $p$} \}, \\
				S_1 & \ := \ \{qp^2 + c_3p^3+c_4p^4+\cdots+c_{k-1}p^{k-1} : c_i \in \Z/p\Z, \; q \neq 0 \text{ a quadratic residue mod $p$} \}, \\
				S_2 & \ := \ \{qp^4 + c_5p^5+c_6p^6+\cdots+c_{k-1}p^{k-1} : c_i \in \Z/p\Z, \; q \neq 0\text{ a quadratic residue mod $p$} \}, \\
				& \quad \vdots \\
				S_{\lceil k/2 \rceil - 1} & \ := \ \{ qp^{2\lceil k/2 \rceil - 2} + cp^{2\lceil k/2 \rceil - 1} : c \in \Z/p\Z, \; q \neq 0 \text{ a quadratic residue mod $p$} \}, \\
				S_{\lceil k/2 \rceil} & \ := \ \{ 0 \}.
			\end{align*}
			Now, for each $n$, every $x \in D_n$ has $x^2 \in S_n$.
			
			To count the multiplicities, first note that all $p^{k-\lceil k/2 \rceil}$ elements of $D_{\lceil k/2 \rceil}$ square to 0. For the rest, we need a lemma:
			\begin{lem}\label{Dn to Sn lemma} Fix $n \in \{0, 1, \ldots, \lceil k/2 \rceil - 1 \}$. For each element in $S_n$, there exist exactly $2p^n$ elements of $D_n$ which square to it.
			\end{lem}
			\begin{proof}\renewcommand{\qedsymbol}{} Consider an element $a = a_np^n + a_{n+1}p^{n+1} + \cdots + a_{k-1}p^{k-1}$ in $D_n$, so $a_n \neq 0$ and all $a_i \in \Z/p\Z$. We make two observations. First, the square of $a$ 
				is
				\be
				a^2 \ \equiv \ \sum_{i=2n}^{k-1}\left( \sum_{j=n}^{i-n} a_ja_{i-j}\right) p^i \quad \bmod p^k.
				\ee
				From this, we observe that the $n$ coefficients $a_{k-n}, \ldots, a_{k-1}$ play no role in determining $a^2$ modulo $p^k$, so at least $p^n$ elements in $D_n$ square to $a^2$.
				
				Second, with $a = a_np^n + a_{n+1}p^{n+1} + \cdots + a_{k-1}p^{k-1}$ as before, consider the element $b = b_np^n + b_{n+1}p^{n+1}+\cdots + b_{k-1}p^{k-1}$ in $D_n$, where we set $b_n := p - a_n$ and $b_i := p - (a_i + 1)$ for all other $i$. We observe that
				\be
				b \ = \ (p-a_n)p^n + \sum_{i=n+1}^{k-1} (p-(a_{i}+1))p^i \ = \ - \left( \sum_{i=n}^{k-1}a_ip^i \right) + p^k \ = \ p^k-a,
				\ee
				so that $b^2 \equiv a^2 \bmod p^k$. Thus our two observations show that at least $2p^n$ elements in $D_n$ square to the element $a^2 \in S_n$.
				
				After computing the cardinalities $|D_n| = p^{k-n}-p^{k-n-1}$ and $|S_n| = \frac{1}{2}(p^{k-2n}-p^{k-2n-1})$, we see that the equation $|D_n| = 2p^n|S_n|$ holds. To finish the lemma, it suffices to show the squaring map is surjective from $D_n$ onto $S_n$, because the existence of some element of $S_n$ with more than $2p^n$ square roots in $D_n$ would contradict this cardinality relation. We will actually show that the squaring map is injective on a subset of $D_n$ with $|S_n|$ elements. Since $(\Z/p^{k-2n}\Z)^\times$ is cyclic, it can be written in terms of a generator $g$ as $\{1, g, g^2, \ldots, g^{p^{k-2n}-p^{k-2n-1}-1}\}$. Let $i, j \in \{0,1,\ldots, |S_n|-1\}$ be distinct. Then the elements $p^ng^i$ and $p^ng^j$, interpreted in $\Z/p^k\Z$, belong to $D_n$. Moreover, their squares $p^{2n}g^{2i}$ and $p^{2n}g^{2j}$ are distinct in $S_n$: Otherwise, we would have $p^{2n}(g^{2i}-g^{2j}) \equiv 0 \bmod p^k$, which holds only if $g^{2i}-g^{2j} \equiv p^{k-2n}a \bmod p^k$ for some $a \in \Z/p^k\Z$. The latter equation implies that $g^{2i}-g^{2j} \equiv 0 \bmod p^{k-2n}$, contradicting that $g^{2i}$ and $g^{2j}$ are different elements in $(\Z/p^{k-2n}\Z)^\times$.
			\end{proof}
			We have determined the images with multiplicity of the squaring map on each part of the domain. Planning to justify afterwards, we compute
			\begin{align}
				\frac{1}{p^k} \sum_{n=1}^{p^k} & \mu(A \cap T^{n^2}A) \ = \ \frac{1}{p^k} \sum_{n=0}^{\lceil k/2 \rceil } \sum_{i \in D_n} \mu(A \cap T^{i^2}A) \\
				& \label{stack 1} = \ \frac{p^{k-\lceil k/2 \rceil}}{p^k}\mu(A) + \frac{1}{p^k} \sum_{n=0}^{\lceil k/2 \rceil - 1} 2p^n \sum_{i \in S_n} \mu(A \cap T^i A)\\
				& \label{stack 2} = \ \frac{1}{p^{\lceil k/2 \rceil}}\mu(A) + \sum_{n=0}^{\lceil k/2 \rceil - 1} \frac{1}{p^{k-n}} \sum_{i \in S_n \cup -S_n} \mu(A \cap T^i A) \\
				& \label{stack 3} = \ \frac{1}{p^{\lceil k/2 \rceil}}\mu(A) + \sum_{n=0}^{\lceil k/2 \rceil - 1} \frac{1}{p^{k-n}} \left( \sum_{j\in \Z/p^{k-2n}\Z} \mu(A \cap T^{jp^{2n}}A) \; - \sum_{j \in \Z/p^{k-2n-1}\Z} \mu(A \cap T^{jp^{2n+1}}A) \right) \\
				& \label{stack 4} = \ \frac{1}{p^{\lceil k/2 \rceil}}\mu(A) + \sum_{n=0}^{\lceil k/2 \rceil - 1} \left( \frac{1}{p^n} \ \langle \E (1_A\vert \calb_{2n}),1_A \rangle - \frac{1}{p^{n+1}} \ \langle \E (1_A\vert \calb_{2n+1}),1_A \rangle \right) \\
				& = \ \frac{1}{p^{\lceil k/2 \rceil}}\mu(A) + \sum_{m=0}^{2\lceil k/2 \rceil - 1} \frac{(-1)^m}{p^{\lceil m/2\rceil}} \ \langle \E (1_A \vert \calb_m ), 1_A\rangle.
			\end{align}
			Equality \eqref{stack 1} holds by Lemma \ref{Dn to Sn lemma}. Equality \eqref{stack 2} holds since $\mu(A \cap T^{i}A) = \mu(A \cap T^{p^k - i}A)$ by shift-invariance and since $q$ is a nonzero quadratic residue mod $p$ iff $p-q$ is not. This idea appeared in the proof of Proposition \ref{p3mod4}. Equality \eqref{stack 3} follows from the observation that
			\ba
			S_n \cup -S_n \ & = \ \{ c_{2n}p^{2n} + c_{2n+1}p^{2n+1} + \cdots + c_{k-1}p^{k-1} : c_i \in \Z/p\Z, c_{2n} \neq 0 \} \\
			&  = \ \{ \text{multiples of $p^{2n}$ that are not multiples of $p^{2n+1}$} \}.
			\ea
			Equality \eqref{stack 4} holds since for every $m \in \{0, 1, \ldots, k\}$, every $x \in \Z/p^k\Z$, and every $A \subset \Z/p^k\Z$, we have
			\be
			\E ( 1_A \vert \calb_m )(x) \ = \ \frac{1}{p^{k-m}} \sum_{j \in \Z/p^{k-m}\Z} 1_A(x-jp^m),
			\ee
			which implies
			\be
			\langle \E (1_A \vert \calb_m), 1_A \rangle \ = \ \frac{1}{p^{k-m}} \sum_{j \in \Z/p^{k-m}\Z} \mu(A \cap T^{jp^m}A).
			\ee
			We consider each parity of $k$. When $k$ is odd, we have $2\lceil k/2\rceil - 1 = k$ and $\calb_k = \calp (Z/p^k\Z)$, so $\langle \E (1_A \vert \calb_k),1_A\rangle = \langle 1_A, 1_A \rangle = \mu(A)$ and the last summand $(m = 2\lceil k/2\rceil - 1)$ becomes
			\be
			\frac{(-1)^k}{p^{\lceil k/2 \rceil}} \langle \E (1_A \vert \calb_k),1_A\rangle \ = \ \frac{-1}{p^{\lceil k/2 \rceil}} \mu(A),
			\ee
			canceling the corresponding term before the sum. When $k$ is even, $2 \lceil k/2 \rceil -1 = k - 1$. In either case, since $\calb_0 = \{ \varnothing, \Z/p^k\Z\}$, we have $\E ( 1_A \vert \calb_0 ) = \mu(A)$, so the first summand $(m=0)$ is $\langle \mu(A), 1_A \rangle = \mu(A)^2$. The formulas follow.
		\end{proof}
		A few comments are in order. First, Propositions \ref{p3mod4} and \ref{p3mod4powers} respectively show that the sequence $( \mathsf{X}_p ), \ p = 3, 7, 11, 19, \ldots$, satisfies Property $an^2$-LA for any nonzero integer $a$ and that, for any integer $k > 1$, the sequences $( \mathsf{X}_{p^k}) \ p = 3, 7, 11, 19, \ldots$, satisfy Property $n^2$-LA. Second, although both formulas \eqref{basic cyclic group results eqn 1} and \eqref{pkgoal} are easy to state, one should not expect such formulas for general polynomials, for distinct sets $A \neq B$, or even for the other residue class of primes $p \equiv 1 \bmod 4$.
		Moreover, on observing that the first formula already shows that the averages of polynomial shifts by $n^2$ equal the optimal value of $\mu(A)^2$ without needing $p$ to grow, one might be tempted to consider a ``stationary'' version of Property $P$-LA, asking for which polynomials $P$ \eqref{basic cyclic group results eqn 1} is true without relying, as in the second proposition, on the asymptotic parameter to smooth it out.
		There may be some life in this question, as \eqref{basic cyclic group results eqn 1} demonstrates that $n^2$, a polynomial that does not merely permute the elements of $\Z/p\Z$, yields the optimal value. A more precise discussion can be found in the appendix.
		
		\section*{Appendix}
		
		In this appendix, we will address the following question:
		
		Fix a prime $p > 2$, and consider the measure-preserving system $\mathsf{X}_p := (\Z/p\Z,\calp(\Z/p\Z),\mu,T)$, where $\mu$ is normalized counting measure and $T$ the map $n \mapsto n+1$ modulo $p$. Which polynomials $q(n) \in \Z[n]$ have the property that for all $A \subset \Z/p\Z$, one has
		\begin{equation}\label{appendix eqn}
			\psum n \mu(A \cap T^{q(n)}A) \ = \ \mu(A)^2?
		\end{equation}
		
		It turns out that the exact condition $q(n)$ must satisfy is that
		\be\label{appendix condition}  \text{For all nonzero } m \in \Z/p\Z, \; |\{ l \in \Z/p\Z : q(l) = \pm m\}| = 2. \ee
		
		Indeed, suppose $q(n)$ satisfies \eqref{appendix condition}. Then $q(n)$ has exactly one zero. Let $S_1 = \{n \in \Z/p\Z : |q^{-1}(\{n\})| = 1\}$, $S_2 = \{n \in \Z/p\Z : |q^{-1}(\{n\})| = 2 \}$, and $-S_2 = \{n : -n \in S_2\}$. For any $n \in S_2$, by $T$-invariance of $\mu$, we have that $2\mu(A \cap T^nA) = \mu(A \cap T^n A) + \mu(T^{-n}A \cap A)$. Thus, since $S_1 \sqcup S_2 \sqcup -S_2 = \Z/p\Z$, we conclude that
		\ba
		\psum n \mu(A \cap T^{q(n)}A) \ & = \ \frac{1}{p} \sum_{n \in S_1} \mu(A \cap T^{n}A) + \frac{2}{p} \sum_{n \in S_2} \mu(A \cap T^n A) \\
		& = \ \frac{1}{p} \sum_{n \in S_1} \mu(A \cap T^{n}A) + \frac{1}{p} \sum_{n \in S_2} \mu(A \cap T^n A) + \frac{1}{p} \sum_{n\in -S_2} \mu(A \cap T^n A) \\
		& = \ \psum n \mu(A \cap T^n A) \\ & = \ \mu(A)^2.
		\ea
		
		Now, suppose that $q(n)$ does not satisfy the condition; namely, there exists nonzero $m \in \Z/p\Z$ such that $i := |\{ l : q(l) = \pm m \}| \neq 2$. For the set $A = \{0,m\}$, one computes easily that
		\[
		\psum n \mu(A \cap T^{q(n)}A) \ = \ \frac{i}{p} \cdot \frac{1}{p} + \frac{|q^{-1}(\{0\})|}{p} \cdot \frac{2}{p} \ = \ \frac{i + 2|q^{-1}(\{0\})|}{p^2}.
		\]
		If $|q^{-1}(\{0\})| \neq 1$, then the set $B = \{0\}$ satisfies
		\[ \psum n \mu(B \cap T^{q(n)} B) \ \neq \ \mu(B)^2.\]
		Indeed, if $|q^{-1}(\{0\})| = 0$, then
		\[ \psum n \mu(B \cap T^{q(n)} B) \ = \ 0 \ < \ \mu(B)^2,\]
		and if $|q^{-1}(\{0\})| > 1$, then
		\[ \psum n \mu(B \cap T^{q(n)} B) \ > \ \mu(B)^2. \]
		Otherwise, we have $|q^{-1}(\{0\})| = 1$, which implies that $i + 2|q^{-1}(\{0\})| \neq 4$ for any $i \neq 2$; hence,
		\[
		\psum n \mu(A \cap T^{q(n)}A) \ \neq \ \mu(A)^2.
		\]
		Indeed, if $i < 2$, then
		\[ \psum n \mu(A \cap T^{q(n)}A) \ < \ \mu(A)^2,\]
		and if $i > 2$, then
		\[ \psum n \mu(A \cap T^{q(n)}A) \ > \ \mu(A)^2.\]
		
		If $p \equiv 3 \bmod 4$, then the map $q(n) = an^2$ satisfies condition \eqref{appendix condition}. This explains Proposition~\ref{p3mod4}.
		

		Along the same lines as the question we have just considered, one may ask the following harder question. Namely, for which polynomials $q(n) \in \Z[n]$ is it true that for any prime $p > 3$, for any set $A$, the equation \eqref{appendix eqn} holds?
		
		A polynomial $q(n) \in \Z[n]$ that is a permutation over $\Z/p\Z$ for every $p$ is certainly an example of the kind of polynomial the question is about. But there are almost no such polynomials. Indeed, suppose $q(n)$ has degree $d > 1$ and leading coefficient $c$. By Dirichlet's theorem on primes in arithmetic progressions, choose a prime $p$ such that $p > |c|$ and $d$ divides $p-1$. Then $q(n)$ has degree $d$ when viewed as a map $\Z/p\Z \to \Z/p\Z$ and hence by \cite[Corollary 7.5]{LN} is not a permutation of $\Z/p\Z$.
		
		Also, such a polynomial $q(n)$ must satisfy the condition~\eqref{appendix condition} for each prime $p$. We do not know whether there is any $q(n)$ which always satisfies that condition and has degree larger than 1. We included the discussion here as an interesting curiosity.
		
		\section*{Acknowledgments} We thank the referee for many useful comments.

	\end{document}